\documentclass[reqno]{amsart}

\numberwithin{equation}{section}
\usepackage[latin1]{inputenc}
\usepackage[english]{babel}

\usepackage{amsmath,amsthm,amsfonts,latexsym,amssymb}
\usepackage[colorlinks]{hyperref}
\hypersetup{linkcolor=blue,citecolor=blue,filecolor=black,urlcolor=blue}
\usepackage{comment}

\def\Xint#1{\mathchoice 
  {\XXint\displaystyle\textstyle{#1}}% 
  {\XXint\textstyle\scriptstyle{#1}}% 
  {\XXint\scriptstyle\scriptscriptstyle{#1}}% 
  {\XXint\scriptscriptstyle\scriptscriptstyle{#1}}% 
  \!\int} 
\def\XXint#1#2#3{{\setbox0=\hbox{$#1{#2#3}{\int}$} 
  \vcenter{\hbox{$#2#3$}}\kern-.5\wd0}} 
\def\Mint{\Xint -}

{ \theoremstyle{plain}
\newtheorem{theorem}{Theorem}[section]
\newtheorem{proposition}[theorem]{Proposition}
\newtheorem{lemma}[theorem]{Lemma}

  \theoremstyle{remark}
\newtheorem{remark}[theorem]{Remark}
  \theoremstyle{definition}
\newtheorem{definition}[theorem]{Definition}
}

\def\R{\mathbb{R}}

\begin{document}
\subjclass[2010]{35B40, 35B65, 35J62, 35Q60, 78A30}

\keywords{Born-Infeld equation, Nonlinear electromagnetism, Mean curvature operator in the Lorentz-Minkowski space, Inhomogeneous quasilinear equation}

\title[On the electrostatic Born-Infeld equation]{On the Born-Infeld equation for electrostatic fields with a superposition of point charges}

\author[D. Bonheure]{Denis Bonheure}
\address{Denis Bonheure and Francesca Colasuonno\newline\indent
D\'epartement de Math\'ematique
\newline\indent Universit\'e Libre de Bruxelles
\newline\indent Campus de la Plaine - CP214
\newline\indent boulevard du Triomphe - 1050 Bruxelles, Belgique
}
\email{denis.bonheure@ulb.ac.be}
\email{francesca.colasuonno@unibo.it}

\author[F. Colasuonno]{Francesca Colasuonno}
%\address{\newline\indent
%D\'epartement de Math\'ematique\newline\indent
%Universit\'e Libre de Bruxelles
%\newline\indent
%Campus de la Plaine - CP214\newline\indent
%boulevard du Triomphe - 1050 Bruxelles, Belgique
%}

\author[J. F\"oldes]{Juraj F\"oldes}
\address{Juraj F\" oldes \newline\indent
Department of Mathematics 
\newline\indent University of Virginia
\newline\indent 141 Cabell Drive, Kerchof Hall
\newline\indent Charlottesville, Virginia 22904, USA
 }
\email{foldes@virginia.edu}

\date{\today}

\begin{abstract} 
In this paper, we study the static Born-Infeld equation
$$
-\mathrm{div}\left(\frac{\nabla u}{\sqrt{1-|\nabla u|^2}}\right)=\sum_{k=1}^n a_k\delta_{x_k}\quad\mbox{in }\mathbb R^N,\qquad \lim_{|x|\to\infty}u(x)=0,
$$ 
where $N\ge3$, $a_k\in\mathbb R$ for all $k=1,\dots,n$, $x_k\in\mathbb R^N$ are the positions of the point charges, possibly non symmetrically distributed, and $\delta_{x_k}$ is the Dirac delta distribution centered at $x_k$. For this problem, we give explicit quantitative sufficient conditions on $a_k$ and $x_k$ to guarantee that the minimizer of the energy functional associated to the problem solves the associated Euler-Lagrange equation. Furthermore, we provide a more rigorous proof of some previous results on the nature of the singularities of the minimizer at the points $x_k$'s depending on the sign of charges $a_k$'s. 
For every $m\in\mathbb N$, we also consider the approximated problem 
$$
-\sum_{h=1}^m\alpha_h\Delta_{2h}u=\sum_{k=1}^n a_k\delta_{x_k}\quad\mbox{in }\mathbb R^N, \qquad\lim_{|x|\to\infty}u(x)=0
$$
where the differential operator is replaced by its Taylor expansion of order $2m$, see \eqref{series}. It is known that each of these problems has a unique solution. We study the regularity of the approximating solution, the nature of its singularities, and the asymptotic behavior of the solution and of its gradient near the singularities. %The study of this approximated problem gives rise to challenging open questions. 
\end{abstract}

\maketitle

\section{Introduction}
The classical electrostatic Maxwell equations in the vacuum lead to the following relations for the electric field: 
\begin{equation}\label{E}
\bold E=-\nabla u, \qquad-\Delta u=\varrho,
\end{equation}
where $\varrho$ is the charge density, $u$ the electric potential, and $\bold E$ the electric field. However, in physically relevant cases when $\rho$ is only an $L^1$-function, or in the case of point charges, the 
model violates the Principle of Finiteness of the energy, see \cite{Feynman, fortunato2002born} for a counterexample. In \cite{BI}, Born and Infeld proposed a nonlinear theory of electromagnetism by modifying Maxwell's equation mimicking  Einstein's special relativity. They introduced a parameter $b\gg 1$, whose inverse is proportional to the radius of the electron, and replaced the Maxwellian Lagrangian density $\mathcal L_M:=\frac12|\bold E|^2$ by 
$$\mathcal L_{BI}:=b^2\left(1-\sqrt{1-\frac{|\bold E|^2}{b^2}}\right)\quad\mbox{for }|\bold E|\le b,$$
so that $\mathcal L_M$ is a first order approximation of $\mathcal L_{BI}$ as  $|\bold E|/b\to 0$.
In presence of a charge density $\varrho$, this new Lagragian leads, at least formally, to replace Poisson's equation in \eqref{E} by the nonlinear equation 
$$-\mathrm{div}\left(\frac{\nabla u}{\sqrt{1-|\nabla u|^2/b^2}}\right) =\varrho,$$
which agrees with the finiteness of the energy even when $\varrho$ is a point charge or an $L^1$-density.
After scaling $u/b$ and $\varrho/b$, we get 
\begin{equation}\label{BI}
-Qu:=-\mathrm{div}\left(\frac{\nabla u}{\sqrt{1-|\nabla u|^2}}\right) =\varrho.
\end{equation}  
It is interesting to notice that the nonlinear operator in \eqref{BI} has also a geometric interpretation, see \cite{BartSim, Ecker}. Indeed $Q$ is the so-called mean curvature operator in the Lorentz-Minkowski space and \eqref{BI} can be seen as the equation for hypersurfaces in Minkowski space with prescribed mean curvature $\rho$. In particular, when $\varrho$ is a superposition of point charges, \eqref{BI} is the equation for area maximizing hypersurfaces in Minkowski space having isolated singularities, cf. \cite{Ecker}. Since the density $\rho$ is not smooth, we look for weak solutions in the space
\begin{equation} \label{def:X}
\mathcal X:=\mathcal D^{1,2}(\mathbb R^N)\cap\{u\in C^{0,1}(\mathbb R^N)\,:\,\|\nabla u\|_\infty\le1\}
\end{equation}
endowed with the norm 
$$
\|u\|:=\left(\int_{\mathbb R^N}|\nabla u|^2 dx\right)^{1/2}.
$$
\begin{definition}\label{def:solution} 
A {\it weak solution} of \eqref{BI} coupled with  the decay condition $$\lim_{|x|\to\infty}u(x)=0$$ is a function $u\in \mathcal X$ such that 
$$
\int_{\mathbb R^N}\frac{\nabla u\cdot\nabla v}{\sqrt{1-|\nabla u|^2}}dx=\langle\varrho,v\rangle\quad\mbox{for all }v\in\mathcal X.
$$
\end{definition}
We recall that $\mathcal D^{1,2}(\mathbb R^N):=\overline{C_{\mathrm{c}}^\infty(\mathbb R^N)}^{\|\cdot\|}$, that is, $\mathcal D^{1,2}(\mathbb R^N)$ is the closure of the space of smooth compactly supported functions with respect to the norm $\|\cdot\|$.
Mathematically,  \eqref{BI} has a variational structure, since it can be (at least formally) seen as the Euler-Lagrange equation of the energy functional 
$I_\varrho:\mathcal X\to\mathbb R$  defined by
\begin{equation}\label{gen_funct}
I_\varrho(u):=\int_{\mathbb R^N}(1-\sqrt{1-|\nabla u|^2})dx-\langle\varrho,u\rangle\quad\mbox{for all }u\in\mathcal X.
\end{equation}
We also denote the dual space of $\mathcal X$ by $\mathcal X^*$ with respect to $L^2(\mathbb R^N)$ inner product, and  we write $\langle\cdot,\cdot\rangle$ for the dual pairing between $\mathcal X^*$ and $\mathcal X$.
It is known  that $I_\varrho$ has a unique minimizer $u_\varrho$ for all $\varrho\in \mathcal{X}^*$ (cf. \cite{BDP} and Section~\ref{Sec2}).
However, due to the lack of regularity of $I_\varrho$ on functions $u$ such that $|\nabla u(x)|=1$ for some points $x\in\R^N$, the justification that minimizers of \eqref{gen_funct} are also weak solutions of \eqref{BI}
presents many difficulties, which will be partly addressed in the present paper.
We remark that some variational problems with a gradient constraint present similar difficulties, see e.g.\cite{brezis1971equivalence,caffarelli1979free,Cellina,treu2000equivalence}. In those papers, the main idea is to remove the constraint on the gradient by defining an appropriate obstacle problem. We believe that some ideas from those papers could be useful in our context but we do not push further those ideas here.      

To address the lack of smoothness,  Bonheure et al.  \cite{BDP} used classical methods from Non-smooth Analysis and weakened the definition of critical point of $I_\varrho$, using the notion critical points in the \textit{weak sense}, see \cite{Szulkin}. Also, they proved the existence and uniqueness of a critical point of $I_\varrho$ in the weak sense, and showed that the PDE is weakly satisfied in the sense of Definition \ref{def:solution} for radially symmetric or locally bounded $\varrho$'s. 

In \cite{fortunato2002born}, Fortunato et al. studied \eqref{BI} in $\mathbb R^3$ and its second-order approximation (by taking the Taylor expansion of the Lagrangian density). In the same spirit, in \cite{Kiessling,BDP} the authors performed higher-order expansions of the Lagrangian density, so that, in the limit, the operator $Q$ can be formally seen as the  series of $2h$-Laplacians
\begin{equation}\label{sum_operator}
-Qu=-\sum_{h=1}^\infty \alpha_h\Delta_{2h}u,
\end{equation}
where we refer to Section \ref{Sec2} for the precise expression of the coefficients and $\Delta_{p}u := \textrm{div}(|\nabla u|^{p - 2}\nabla u)$. This expension allows to approximate $Q$ with the operators sum 
\begin{equation}\label{approxQ}
-\sum_{h=1}^m \alpha_h\Delta_{2h}
\end{equation}
and \eqref{BI} with the quasi-linear equations
$$
-\sum_{h=1}^m \alpha_h\Delta_{2h}\phi=\varrho\qquad	\mbox{for }m\in\mathbb N.
$$
Each of such equations, complemented with the condition $\lim_{|x|\to\infty}u(x)=0$, have a unique solution $u_m$. 
In \cite{Kiessling}, respectively \cite{BDP}, it is further proved that the approximating solutions $u_m$'s weakly converge to the minimizer $u_\varrho$ of \eqref{BI} when $\rho$ is a superposition of point charges, respectively for any $\varrho\in\mathcal X^*$. 

It is worth noting that $\mathcal X^*$ contains Radon measures and in particular superpositions of point charges and $L^1$-densities, which are in turn dense in the space of 
Radon measures.
Due to these reasons we will assume that $\rho$ is  a finite superposition of charges without any symmetry conditions, that is, we consider
\begin{equation}\label{P}
\begin{cases}\displaystyle{-\mathrm{div}\left(\frac{\nabla u}{\sqrt{1-|\nabla u|^2}}\right)}=\sum_{k=1}^n a_k\delta_{x_k} \quad \mbox{in } \mathbb R^N,\\
\displaystyle{\lim_{|x|\to\infty}u(x)}=0,
\end{cases}
\end{equation}
where $N\ge3$, $\delta_{x_k}$ is the Dirac delta function centered at $x_k$, $a_k\in\mathbb R$ and $x_k\in\mathbb R^N$ for $k=1,\dots,n$. This situation is 
general enough to cover most of the phenomena, yet simple enough that it can be analyzed explicitly.  
The energy functional associated to \eqref{P} has the form
\begin{equation}\label{I-superposition}
I(u)=\int_{\mathbb R^N}(1-\sqrt{1-|\nabla u|^2})dx-\sum_{k=1}^n a_k u(x_k)\quad\mbox{for all }u\in \mathcal X.
\end{equation}
Problem \eqref{P} has been first studied in \cite{Kiessling,BDP}, see also Section \ref{Sec2} below, where we report some recalls. 

Our first goal is to provide a rigorous proof concerning the nature of the singularities $x_k$'s for the minimizer $u_\varrho$ of $I$, depending on the sign of the charges $a_k$'s, see \cite{Kiessling} and Theorem \ref{maxmin} below. 
More precisely, in Theorem \ref{maxmin}, we show that if the charge $a_k$ is positive (resp. negative) then the point charge $x_k$ is a relative strict maximizer (resp. minimizer) for $u_\varrho$. Our proof uses 
geometric results proved by Ecker \cite{Ecker} and the comparison principle in bounded domains proved in Lemma \ref{comparison}. This result is far from obvious, since $u_\varrho$ is globally bounded 
and in particular it does not diverge at $x_k$, rather $\nabla u_\varrho$ is discontinuous at the location of the charges. Of course since the problem is not linear it cannot be decomposed into 
several problems, each with just one point charge. However, this is not the only obstacle, if one replaces our curvature operator with Laplacian in one dimension, then the Green's function for the charge
located at $x_k$ has the form $|x - x_k|$
and in particular it is bounded in the neighborhood of $x_k$. But, adding several Green's functions one obtains that the solution is a piece-wise linear function, which might not have local extrema at $x_k$. 
Although the singularity is of the same nature as one for Laplacian in one dimension, it is crucial that the solution vanishes at infinity, which introduces a non-local argument into the proofs.

We immediately show an application of these results in the question
whether the minimizer $u_\varrho$ of \eqref{I-superposition} is a weak solution of \eqref{P}. To our best knowledge, this problem hasn't been completely solved yet. Some results in this direction can be found in \cite{Kiessling}, but the main arguments in that paper need to be adjusted (see the discussion in \cite[Section 4]{BDP}). To our knowledge, the case of a generic $\varrho$ is still open.  
In \cite{Kiessling,BDP}, the authors proved that $u_\varrho$ solves the equation in \eqref{P}, in  $\mathbb R^N\setminus\Gamma$, where $\Gamma:=\bigcup_{k\neq j}\overline{x_k x_j}$ and $\overline{x_k x_j}$ denotes the line segment with endpoints $x_k$ and $x_j$. Furthermore, it is proved in \cite{BDP} that if the charges are sufficiently small or far apart, $u_\varrho$ solves the equation in $\mathbb R^N\setminus\{x_1,\dots,x_n\}$. In particular, in \cite{Kiessling} it is showed that if two point charges $x_k$, $x_j$ have the same sign $a_k\cdot a_j>0$, then $u_\varrho$ solves the equation also along the open line segment $\mathrm{Int}(\overline{x_k x_j})$.  

The arguments on the literature are based on the fact that, if the minimizer does not satisfy the equation along the segment connecting $x_k$ and $x_j$, then it must be affine and since the minimizer is bounded, 
then one obtains a contradiction. However, the argument is purely qualitative and it does not yield an easily verifiable condition based only on the location and strength of the charges. 
In this paper, we partly bridge this gap by proving a sufficient {\it quantitative} condition on the charges and on their mutual distance to guarantee that the minimizer $u_\varrho$ solves \eqref{P} also along the line segments joining two charges of different sign. Let us denote $\mathcal K_+ := \{k : a_k > 0\}$ and $\mathcal K_- := \{k : a_k < 0\}$, that is, set of indexes for positive respectively negative charges.   Our result reads then as follows. 

\begin{theorem}\label{prop>2charges} 
If 
\begin{equation}\label{hpjl}
\left(\frac{N}{\omega_{N-1}}\right)^{\frac{1}{N-1}}\frac{N-1}{N-2}\left[\left(\sum_{k\in \mathcal K_+}a_k\right)^{\frac{1}{N-1}}+\left(\sum_{k\in \mathcal K_-}|a_k|\right)^{\frac{1}{N-1}}\right] <\min_{\underset{j\neq l}{j,\,l\in\{1,\dots,n\}}} |x_j-x_l|,
\end{equation}
where $\omega_{N-1}$ is the measure of the unit sphere in $\mathbb R^N$, then $$u_\varrho\in C^\infty(\mathbb R^N\setminus\{x_1,\dots, x_n\})\cap C(\mathbb R^N)$$ and it is a 
classical solution of \eqref{general} in $\mathbb R^N\setminus\{x_1,\dots, x_n\}$, with $|\nabla u_\varrho|<1$. 
%Let $j\neq l\in\{1,\dots,n\}$ be such that $a_j\cdot a_l<0$. If 
%\begin{equation}\label{hpjl}
%\left(\frac{N}{\omega_{N-1}}\right)^{\frac{1}{N-1}}\frac{N-1}{N-2}\left[\left(\sum_{k\in \mathcal K_+}a_k\right)^{\frac{1}{N-1}}+\left(\sum_{k\in \mathcal K_-}|a_k|\right)^{\frac{1}{N-1}}\right] < |x_j-x_l|,
%\end{equation}
%where $\omega_{N-1}$ is the measure of the unit sphere in $\mathbb R^N$, then $u$ is smooth along the open segment $ \mathrm{Int}(\overline{x_j x_l})$, that is, 
%$$u_\varrho\in C^\infty((\mathbb R^N\setminus\Gamma) \cup \mathrm{Int}(\overline{x_j x_l}))\cap C(\mathbb R^N)$$ and it is a 
%classical solution of \eqref{general}, with $|\nabla u_\varrho|<1$, in $(\mathbb R^N\setminus\Gamma) \cup \mathrm{Int}(\overline{x_j x_l})$
\end{theorem}

Note that the occurrence of the sum of positive and negative charges is natural, since we cannot rule out the situation when these charges are close to each other and they appear as one point
charge. The explicit form of the constant on the left-hand side of \eqref{hpjl} is crucial and observe that is bounded from below independently of $N$ and the number of charges. This allows for passing to 
the limit in the number of charges, the formulation of the result is left to the interested reader. 
We also give in Remark \ref{3.8} a more precise way (although less explicit) how to calculate the constant on the left-hand side of \eqref{hpjl} in the general case, and yet more optimal one if there are only 
{\it two} point charges of different sign in Proposition~\ref{2points}. 
%We remark that it is not believed that the assertion of Proposition \ref{prop>2charges} holds true without assumptions on $(x_k)$ and $(a_k)$. 

The proof of this theorem is based both on a new version of comparison principle (Lemma~\ref{comparison}) and on the explicit expression of the best constant $\bar C$ for the inequality 
$$
\|\nabla u\|^2_{L^2(\mathbb R^N)}\ge \bar C \|u\|^N_{L^\infty(\mathbb R^N)}\quad\mbox{for all }u\in\mathcal X,
$$
proved in Lemmas \ref{firstlemma} and \ref{secondlemma}, which might be of independent interest. Note that this result has a different flavor than the results for optimal constants for the 
embeddings since our  inequality is inhomogeneous and we have to crucially use that the Lipschitz constant of $u$ is bounded by one.

In Section \ref{Sec3}, we first turn our attention to the approximating problems
\begin{equation}\label{Pa}
\begin{cases}\displaystyle{-\sum_{h=1}^m\alpha_h\Delta_{2h}u=\sum_{k=1}^n a_k\delta_{x_k}} \quad \mbox{in } \mathbb R^N,\\
\displaystyle{\lim_{|x|\to\infty}u(x)}=0
\end{cases}
\end{equation}
for $m\ge1$ and study the regularity of the solution $u_m$: by combining results of Lieberman \cite{Lieberman}, a linearization, and a bootstrap argument, we prove that 
the solutions are regular away from the points $x_k$'s. 

\begin{proposition}\label{reg} Let $2m>\max\{N,2^*\}$, $2^*:=2N/(N-2)$, and $u_m$ be the solution of \eqref{Pa}. Then $u_m\in C^{0,\beta_m}_0(\mathbb R^N)\cap C^\infty(\mathbb R^N\setminus\{x_1,\dots,x_n\})$, where 
$$C_0^{0,\beta_m}(\mathbb R^N):=\left\{u\in C^{0,\beta_m}(\mathbb R^N)\,:\,\lim_{|x|\to\infty}u(x)=0\right\},$$
with $\beta_m:=1-\frac{N}{2m}$. 
\end{proposition}

In comparison to the full problem \eqref{P}, there is an important difference -- we do not have a priori an estimate on $|\nabla u|$, and therefore the H\" older estimate is not immediate. 
Note that $\beta_m$ converges to $1$ as $m \to \infty$, in agreement with the fact that the solutions of \eqref{Pa} approximate solutions of \eqref{P}. 
On the other hand the operator  in \eqref{Pa} is well defined for any sufficiently smooth function $u$ and the smoothness of solutions can be expected away from $x_k$'s.

We stress that in the proof of Proposition \ref{reg}, we heavily use the fact that in the sum operator \eqref{approxQ} appears also the Laplacian, see Remark \ref{roleoflapl} for further details. 
Moreover, we also prove that $u_m$ and $\nabla u_m$ behave as the fundamental solution (and its gradient) of the $2m$-Laplacian near the singularities $x_k$'s. Intuitively,  we could say that the Laplacian, $\Delta_2$, 
is responsible for the regularity of the approximating solution $u_m$ and the behavior at infinity, while the $2m$-Laplacian (the last one) dictates the local behavior of the solution $u_m$ near the singularities $x_k$'s, in the following sense. 

\begin{theorem}\label{approx_nabla}
Let $2m>\max\{N,2^*\}$ and $k=1,\dots,n$. Then
\begin{equation}\label{u-growth}
\lim_{x\to x_k}\frac{u_m(x)-u_m(x_k)}{|x-x_k|^{\frac{2m-N}{2m-1}}}=K_m
\end{equation}
for some $K_m=K_m(a_k,\alpha_m,N)\in\mathbb R$ such that $K_m\cdot a_k<0$, and
\begin{equation}\label{grad-growth}
\lim_{x\to x_k}\frac{|\nabla u_m (x)|}{|x-x_k|^{\frac{1-N}{2m-1}}}=K_m',
\end{equation}
with $K_m':=\frac{2m-N}{2m-1}|K_m|$.
In particular, $x_k$ is a relative strict maximizer (resp. minimizer) of $u_m$ if $a_k>0$ (resp. $a_k<0$). 
\end{theorem} 

The same reasons as above make this result non-trivial. The operator is non-linear, thus it is not obvious that the local behavior does not depend on the location of all charges as it for example does for 
the Laplacian in one dimension. The asymptotic behavior is a fine interplay between lowest and highest order differential operators in the expansion. 

The proof of this theorem is rather technical and relies on a blow-up argument, combined with Riesz potential estimates \cite{Baroni}. Such a usage of blow-up method is quite unusual since the solution
is bounded at the blow-up point and we need to rescale the problem in such a way that we keep the boundedness of solution, but remove the lower order terms. 

The fact that the growth rate of $u_m$ near the singularity $x_k$ is of the type $|x-x_k|^{\frac{2m-N}{2m-1}}$, with exponent that goes to 1 as $m$ goes to infinity, shows that the singularities $x_k$'s of $u_m$ approach cone-like singularities for $m$ large, which is coherent with the results found for $u_\varrho$. In particular, we note that the blow up rate \eqref{grad-growth} of $|\nabla u_m|$ near the singularities and the fact that $\lim_{m\to\infty} K_m'=1$ (cf. Remark \ref{Kto1}) suggest that $\lim_{m\to\infty}|\nabla u_m(x)|\approx 1$ as $x\to x_k$, which is the same behavior as $|\nabla u_\varrho|$, see  \cite[Theorem 1.4]{Kiessling}. 
Moreover, as an easy consequence of \eqref{u-growth}, we get that the singularity $x_k$ is either a relative strict minimizer or a relative strict maximizer depending on the sign of its coefficient $a_k$. Altogether, this shows that the approximating solutions $u_m$'s are actually behaving like the minimizer $u_\varrho$ of \eqref{I-superposition}, at least qualitatively near the singularities.

Furthermore, it is worth stressing that problem  \eqref{Pa} is governed by an inhomogeneous operator that behaves like $-\Delta-\Delta_{2m}$ with $m$ large. The interest in inhomogeneous operators of the type sum of a $p$-Laplacian and a $q$-Laplacian has recently significantly increased, as shown by the long list of recent papers, 
see for instance \cite{Mingione, Baroni, ColSq, cupini2014existence, martinez2008minimum, Mihailescu} and the references therein.
 
The paper is organized as follows. In Section \ref{Sec2} we collect definitions and known results for problems \eqref{P} and \eqref{Pa} relevant to our proof. 
Section \ref{Sec4} contains our results concerning the qualitative properties of the minimizer of the original problem \eqref{P} and the  sufficient conditions to guarantee that the minimizer $u_\varrho$ of $I$ indeed solves \eqref{P}. 
Finally, Section \ref{Sec3} is devoted to the study of the approximating problem \eqref{Pa} and the qualitative analysis of the solution and its gradient.  

\section{Preliminaries}\label{Sec2}

In this section we summarize used notation and definitions as well as previous results needed in the rest of the paper. 
We start with properties of functions belonging to the set $\mathcal{X}$, see \eqref{def:X}.  

\begin{lemma}[Lemma 2.1 of \cite{BDP}]\label{lemma2.1} The following properties hold: 
\begin{itemize}
\item[(i)] $\mathcal X\hookrightarrow W^{1,p}(\mathbb R^N)$ for all $p\ge 2^*$;
\item[(ii)] $\mathcal X\hookrightarrow L^\infty(\mathbb R^N)$;
\item[(iii)] If $u\in\mathcal X$, $\lim_{|x|\to\infty}u(x)=0$;
\item[(iv)] $\mathcal X$ is weakly closed;
\item[(v)] If $(u_n)\subset \mathcal X$ is bounded, up to a subsequence it converges weakly to a function $\bar u\in\mathcal X$, uniformly on compact sets.  
\end{itemize}
\end{lemma}

Throughout the paper $\overline{xy} := \big\{z\;:\;z= (1 - t)x + ty\mbox{ for } t \in [0, 1] \big\}$ denotes the line segment with endpoints $x$ and $y$ and $\mathrm{Int}(\overline{xy})$ the open segment. 

\begin{definition} Let $u\in C^{0,1}(\Omega)$, with $\Omega\subset\mathbb R^N$. We say that
\begin{itemize}
\item[(i)] $u$ is {\it weakly spacelike} if $|\nabla u|\le 1$ a.e. in $\Omega$;
\item[(ii)] $u$ is {\it spacelike} if $|u(x)-u(y)|<|x-y|$ for all $x,\,y\in\Omega$, $x\neq y$, and the line segment $\overline{xy}\subset\Omega$;
\item[(iii)] $u$ is {\it strictly spacelike} if $u\in C^1(\Omega)$, and $|\nabla u|<1$ in $\Omega$.
\end{itemize} 
\end{definition} 

\begin{proposition}[Proposition 2.3 of \cite{BDP}]\label{Prop2.3} For any $\varrho\in\mathcal X^*$ there exists a unique $u_\varrho\in\mathcal X$ that minimizes $I_\varrho$ defined by \eqref{gen_funct}. If furthermore $\varrho\neq 0$, then $u_\varrho\neq 0$ and $I_\varrho(u_\varrho)<0$.
\end{proposition}

\begin{theorem}[Theorem 1.6 and Lemma 4.1 of \cite{BDP}]\label{Le4.1}
Let $\varrho:=\sum_{k=1}^n a_k\delta_{x_k}$ and $\Gamma:=\bigcup_{k\neq j}\overline{x_k x_j}$.
The minimizer $u_\varrho$ of the energy functional $I$ given by \eqref{I-superposition} is a strong solution of 
$$
\begin{cases}
-\mathrm{div}\left(\frac{\nabla u}{\sqrt{1-|\nabla u|^2}}\right)=0\quad\mbox{in }\mathbb R^N\setminus\Gamma,\\
\lim_{|x|\to\infty}u(x)=0.
\end{cases}
$$
Furthermore, 
\begin{itemize}
\item[$(i)$] $u_\varrho\in C^\infty(\mathbb R^N\setminus\Gamma)\cap C(\mathbb R^N)$;
\item[$(ii)$] $u_\varrho$ is strictly spacelike in $\mathbb R^N\setminus \Gamma$;
\item[$(iii)$] for $k\neq j$, either $u_\varrho$ is a classical solution on $\mathrm{Int}(\overline{x_k x_j})$, or
$$u_\varrho(tx_k+(1-t)x_j)=tu_\varrho(x_k)+(1-t)u_\varrho(x_j)\quad\mbox{for all }t\in(0,1).$$
\end{itemize}
\end{theorem}

\begin{theorem}[Corollary 3.2 of \cite{Kiessling}]\label{corKiess}
If $a_k\cdot a_j>0$, then $u_\varrho$ is a classical solution on $\mathrm{Int}(\overline{x_k x_j})$.
\end{theorem}

%\alert{We precise that actually the results in Theorem \ref{Le4.1}-$(iii)$ and Theorem \ref{corKiess} do not apply to possible points of intersection among different segments joining different couples of charges.?}

%
%\begin{lemma}[Lemma 2.2 of \cite{BDP}] The functional $I:\mathcal X\to\mathbb R$ is bounded from below, coercive, continuous and strictly convex, hence weakly lower semicontinuous.
%\end{lemma}

%\begin{remark} Due to the lack of regularity of the functional $I$ at points $x$ such that $|\nabla u(x)|=1$, it is not known if such minimizer $u_\varrho$ is a weak solution of \eqref{P} for a general $\varrho\in\mathcal X^*$. 
%\end{remark}

As mentioned in the introduction, in order to overcome the difficulty related to the non-differentiability of $I$, we consider  approximating problems. The idea is to approximate the mean curvature operator $Q$ 
(for the definition see \eqref{sum_operator}) by a finite sum of $2h$-Laplacians, by using the Taylor expansion. We note that the operator $Q$ is formally  
the Fr\' echet derivative of the functional 
\begin{equation}\label{series}
\int_{\mathbb R^N}\left(1-\sqrt{1-|\nabla u|^2}\right)dx=\int_{\mathbb R^N} \sum_{h=1}^\infty\frac{\alpha_h}{2h}|\nabla u|^{2h} dx,
\end{equation}
where $\alpha_1:=1$, $\alpha_h:=\frac{(2h-3)!!}{(2h-2)!!}$ for $h\ge2$, and 
$$
k!!:=\prod_{j=0}^{[k/2]-1}(k-2j) \quad\mbox{for all } k\in\mathbb N.
$$
The series in the right-hand side of \eqref{series} converges pointwise, although not uniformly, for all $|\nabla u|\le 1$. 
Then, the operator $-Qu=-\mathrm{div}\left(\frac{\nabla u}{\sqrt{1-|\nabla u|^2}}\right)$ can be regarded as the series of $2h$-Laplacians, see \eqref{sum_operator}.

For every natural number $m\ge1$, we define the space $\mathcal X_{2m}$ as the completion of $C^\infty_{\mathrm{c}}(\mathbb R^N)$ with respect to the norm 
$$\|u\|_{\mathcal{X}_{2m}}:=\left[\int_{\mathbb{R}^N}|\nabla u|^2 dx+\left(\int_{\mathbb{R}^N}|\nabla u|^{2m}dx\right)^{1/m}\right]^{1/2}.$$

Let $\varrho\in \mathcal X_{2m}^*$ for some $m\ge1$. We study the approximating problem
\begin{equation}\label{Pm}
\begin{cases}
-\displaystyle{\sum_{h=1}^m}\alpha_h\Delta_{2h}u = \varrho \quad \mbox{in }\mathbb R^N,\\
\displaystyle{\lim_{|x|\to\infty}}u(x)=0
\end{cases}
\end{equation}
and we denote by $I_{m}:\mathcal X_{2m}\to \mathbb R$ the energy functional associated to \eqref{Pm}
$$I_{m}(u):=\sum_{h=1}^m\frac{\alpha_h}{2h}\int_{\mathbb R^N}|\nabla u|^{2h}dx-\langle \varrho,u\rangle_{\mathcal X_{2m}},$$
where $\langle \cdot,\cdot\rangle_{\mathcal X_{2m}}$ denotes the duality pairing between $\mathcal X_{2m}^*$ and $\mathcal X_{2m}$. The functional $I_{m}$ is of class $C^1$ and is the $m$th-order approximation of $I$. 

\begin{definition} A weak solution of \eqref{Pm} is a function $u_m\in\mathcal X_{2m}$ such that 
$$\sum_{h=1}^m\alpha_h\int_{\mathbb R^N}|\nabla u_m|^{2h-2}\nabla u_m\nabla vdx=\langle \varrho,v\rangle_{\mathcal X_{2m}}\quad\mbox{for all }v\in C_\mathrm{c}^\infty(\mathbb R^N).$$ 
\end{definition}

Clearly a function is a weak solution of \eqref{Pm} if and only if it is a critical point of $I_m$.

\begin{proposition}[Proposition 5.1 of \cite{BDP}] Let $\varrho\in\mathcal X^*_{2m_0}$ for some $m_0\ge1$. Then, for all $m\ge m_0$, the functional $I_{m} :\mathcal X_{2m}\to\mathbb R$ has one and only one critical point $u_m$. Furthermore, $u_m$ minimizes $I_{m}$. 
\end{proposition}

\begin{theorem}[Theorem 5.2 of \cite{BDP}] Let $\varrho\in\mathcal X^*_{2m_0}$ for some $m_0\ge1$. Then $u_{m}\rightharpoonup u_\varrho$ in $\mathcal X_{2m}$ for all $m\ge m_0$ and uniformly on compact sets.  
\end{theorem}

\section{Born-Infeld problem}
\label{Sec4}

In this section we study  the nature of the singularities of the minimizer of energy functional \eqref{I-superposition}
and sufficient conditions guaranteeing that the minimizer is a solution of \eqref{P} on $\mathbb{R}^N \setminus \{x_1,\dots,x_n\}$. 
To this aim, we isolate one singularity, and we investigate \eqref{P} on bounded domains.  We start with definitions and preliminary results.
 
Let $\Omega \subset \mathbb R^N$ be a bounded domain, $\varphi:\partial\Omega\to\mathbb R$ a bounded function and $\varrho_\Omega\in\mathcal X_\Omega^*$, where $\mathcal X_\Omega^*$ is the dual space of $\mathcal X_\Omega:=\{u\in C^{0,1}(\Omega)\,:\, |\nabla u|\le 1 \mbox{ a.e. in }\Omega\}$. We consider the variational problem
\begin{equation}\label{VP}
\min_{u\in \mathcal C(\varphi,\Omega)}I_{\Omega,\varrho}(u),
\end{equation}
where
$$I_{\Omega,\varrho}(u):=\int_\Omega\left(1-\sqrt{1-|\nabla u|^2}\right)dx-\langle\varrho_\Omega,u\rangle_{\mathcal X_\Omega}\quad\mbox{for all }u\in\mathcal X_\Omega$$
and $$\mathcal C(\varphi,\Omega):=\{v\in\mathcal X_\Omega\,:\,v=\varphi\mbox{ on }\partial\Omega\}.$$

\begin{lemma}\label{uniqueness}
The problem \eqref{VP} has at most one solution.
\end{lemma}

\begin{proof} Although the argument is similar to \cite[Proposition 1.1]{BartSim}, we include it here for completeness. 
Let $u_1,\, u_2\in\mathcal X_\Omega$ be two solutions of \eqref{VP} and consider $u_t:=(1-t)u_1+t u_2$ for any $t\in(0,1)$. By the convexity of $1-\sqrt{1-|x|^2}$, we have 
\begin{equation}\label{Iut}
\begin{aligned}
I_{\Omega,\varrho}(u_t)&\le (1-t)\int_\Omega(1-\sqrt{1-|\nabla u_1|^2})dx+t\int_\Omega(1-\sqrt{1-|\nabla u_2|^2})dx\\
&\quad-(1-t)\langle\varrho_\Omega,u_1\rangle_{\mathcal X_\Omega}-t\langle\varrho_\Omega,u_2\rangle_{\mathcal X_\Omega}\\
&=(1-t)I_{\Omega,\varrho}(u_1)+tI_{\Omega,\varrho}(u_2)=I_{\Omega,\varrho}(u_1),
\end{aligned}
\end{equation}
where we used  $I_{\Omega,\varrho}(u_1)=I_{\Omega,\varrho}(u_2)=\min I_{\Omega,\varrho}$. By the minimality of $I_{\Omega,\varrho}(u_1)$, we have $I(u_t)=I(u_1)$, and so the equality must hold in \eqref{Iut}. Now, being $x \mapsto 1 - \sqrt{1 - |x|^2}$ strictly convex, we have 
$\nabla u_1=\nabla u_2$ a.e. in $\Omega$. Since $u_1=u_2$ on $\partial\Omega$, $u_1-u_2$ can be extended to a Lipschitz function on $\mathbb R^N$ that vanishes in $\mathbb R^N\setminus\Omega$, cf. \cite{BartSim}. 
Thus, being $\nabla (u_1-u_2)=0$ a.e. in $\Omega$, we have $u_1=u_2$ and the proof is concluded.    
\end{proof}

\begin{remark}\label{existence}
Concerning existence of a minimizer for \eqref{VP}, we observe that in the case under consideration $\varrho=\sum_{k=1}^na_k\delta_{x_k}$, it is immediate to see that for every $\Omega\subset \mathbb R^N\setminus\{x_1,\dots,x_n\}$, $u_\varrho|_\Omega$ minimizes $I_\Omega$
over $\mathcal C(u_\rho, \Omega)$, where we recall that $u_\varrho$ denotes the unique minimizer of $I_\varrho$ in all of $\mathbb R^N$, cf. Proposition \ref{Prop2.3}.
Indeed, let
$v \in \mathcal C(u_\rho, \Omega)$ and denote $\psi := v - u_\rho \in \mathcal C(0,\Omega)$ and $\tilde \psi$  Lipschitz continuation of $\psi$ 
that vanishes outside of 
$\Omega$. Then $u_\rho +  \tilde \psi\in\mathcal X$ and
the minimality of $u_\varrho$ yields
$$
\begin{aligned}
I(u_\varrho+\tilde\psi)&=\int_{\Omega}\left(1-\sqrt{1-|\nabla(u_\varrho+\tilde\psi)|^2}\right)dx+\int_{\mathbb R^N\setminus\Omega}\left(1-\sqrt{1-|\nabla u_\varrho|^2}\right)dx\\
&\phantom{=}-\sum_{k=1}^n a_ku_\varrho(x_k)\\
&\ge I(u_\varrho)=\int_{\mathbb R^N}\left(1-\sqrt{1-|\nabla u_\varrho|^2}\right)dx-\sum_{j=1}^n a_ku_\varrho(x_k)\,.
\end{aligned}
$$
Hence, 
$$
\int_\Omega\left(1-\sqrt{1-|\nabla(u_\varrho|_\Omega+\psi)|^2}\right)dx\ge \int_\Omega\left(1-\sqrt{1-|\nabla u_\varrho|_\Omega|^2}\right)dx
$$
or equivalently
$$
I_\Omega(v) = I_\Omega(u_\varrho|_\Omega+\psi)\ge I_\Omega(u_\varrho|_\Omega),
$$
which proves the claim by the arbitrariness of $v\in \mathcal C(u_\varrho,\Omega)$.
\end{remark}
\begin{definition} 
Let $\varrho_1,\,\varrho_2\in\mathcal X_\Omega^*$. We say that $\varrho_1\le\varrho_2$ if  $\langle \varrho_1, v\rangle_{\mathcal X_\Omega}\le\langle\varrho_2,v\rangle_{\mathcal X_\Omega}$ for all $v\in\mathcal X_\Omega$ with $v\ge0$.
\end{definition}

\begin{lemma}\label{comparison}
Let $\varrho_1,\,\varrho_2\in\mathcal X_\Omega^*$, $\varphi_1,\,\varphi_2:\partial\Omega\to\mathbb R$ be two bounded functions, $u_1\in \mathcal C(\varphi_1,\Omega)$ be the minimizer of $I_{\Omega,\varrho_1}$, and $u_2\in \mathcal C(\varphi_2,\Omega)$ be the minimizer of $I_{\Omega,\varrho_2}$. If $\varrho_2\le\varrho_1$, then 
$$u_2(x)\le u_1(x)+\sup_{\partial\Omega}(\varphi_2-\varphi_1)\quad\mbox{for all }x\in\Omega.$$
\end{lemma}

\begin{proof} Throughout this proof we use the following simplified notation $$I_1:=I_{\Omega,\varrho_1},\quad I_2:=I_{\Omega,\varrho_2},\quad\langle\cdot,\cdot\rangle:=\langle\cdot,\cdot\rangle_{\mathcal X_\Omega},\quad \mathcal Q(u):=\int_\Omega(1-\sqrt{1-|\nabla u|^2})dx.$$
Let $\alpha:=\sup_{\partial\Omega}(\varphi_2-\varphi_1)$ and $\tilde u_1:=u_1+\alpha$. We claim that $\tilde u_1$ minimizes $I_1$ in $\mathcal C(\varphi_1+\alpha,\Omega)$. Indeed, since $u_1$ minimizes $I_1$ in $\mathcal C(\varphi_1,\Omega)$, for all $u\in \mathcal C(\varphi_1,\Omega)$ we have 
$$
\begin{aligned}
I_1(\tilde u_1)&=\mathcal Q(u_1)-\langle\varrho_1,u_1\rangle-\langle\varrho_1,\alpha\rangle\le I_1(u)-\langle\varrho_1,\alpha\rangle=I_1(u+\alpha).
\end{aligned}
$$
Since $\mathcal C(\varphi_1+\alpha,\Omega)=\mathcal C(\varphi_1,\Omega)+\alpha$, the claim is proved. 

Now, suppose by contradiction that the set $\Omega^+:=\{x\in\Omega\,:\,u_2(x)>\tilde u_1(x)\}$ is non-empty. 
Let $\Omega^-:=\Omega\setminus\Omega^+$,  
$$
\begin{gathered} \mathcal Q^+(u):=\int_{\Omega^+}(1-\sqrt{1-|\nabla u|^2})dx,\quad \mathcal Q^-(u):=\int_{\Omega^-}(1-\sqrt{1-|\nabla u|^2})dx \,, \\
U:= \max\{u_2, \tilde{u}_1\} =  \begin{cases}
\tilde u_1\quad&\mbox{in }\Omega^-\\
u_2&\mbox{in }\Omega^+,
\end{cases}\quad\mbox{and}\quad
V:= \min\{u_2, \tilde{u}_1\} = \begin{cases}
u_2\quad&\mbox{in }\Omega^-\\
\tilde u_1&\mbox{in }\Omega^+.
\end{cases}
\end{gathered}
$$
We observe that, by continuity, $u_2=\tilde u_1$ on $\partial \Omega^+$. Hence, $U\in \mathcal C(\varphi_1+\alpha,\Omega)$ and $V\in \mathcal C(\varphi_2,\Omega)$. Furthermore, the following relations hold in the whole of $\Omega$:
$$
u_2 - V = U-\tilde u_1\ge 0 .
$$ 
Then, by $\varrho_2\le \varrho_1$, we obtain
$$
\begin{aligned}
I_1(U)&=\mathcal Q(U)-\langle\varrho_1,U-\tilde u_1\rangle -\langle\varrho_1,\tilde{u}_1\rangle\\
&\le \mathcal Q(U)-\langle\varrho_2,U-\tilde u_1\rangle -\langle\varrho_1,\tilde{u}_1\rangle\\
&=\mathcal Q^+(u_2)+\mathcal Q^-(\tilde u_1)-\langle\varrho_2,U-\tilde u_1\rangle -\langle\varrho_1,\tilde{u}_1\rangle\\
&=I_1(\tilde u_1)-\mathcal Q^+(\tilde u_1)+\mathcal Q^+(u_2)-\langle\varrho_2,U-\tilde{u}_1\rangle\\
&=I_1(\tilde u_1)+I_2(u_2)-\mathcal Q^-(u_2)+\langle\varrho_2,u_2\rangle-\mathcal Q^+(\tilde u_1)-\langle\varrho_2,U-\tilde u_1\rangle\\
&=I_1(\tilde u_1)+I_2(u_2)-\mathcal Q(V) + \langle\varrho_2, V\rangle \\
&=I_1(\tilde u_1)+I_2(u_2)-I_2(V)\\
&<I_1(\tilde u_1),
\end{aligned}
$$
where in the last step we used the strict minimality of $I_2(u_2)$ over $\mathcal C(\varphi_2,\Omega)$, see Lemma \ref{uniqueness}. 
This contradicts the fact that $\tilde u_1$ minimizes $I_1$ in $\mathcal C(\varphi_1+\alpha,\Omega)$ and concludes the proof.
\end{proof}

\begin{theorem}\label{maxmin}
If $u_\varrho$ is the unique minimizer of the problem \eqref{VP}, then for every $k=1,\cdots,n$ one has
\begin{itemize}
\item[$(i)$] For every $x \in \mathbb{R}^N$ with $|x| = 1$, there exists $\displaystyle{\lim_{h\to 0^+}\frac{u_{\varrho}(hx+x_k)-u_{\varrho}(x_k)}h}= \pm 1$;
\item[$(ii)$] $x_k$ is a relative strict minimizer (resp. maximizer) of $u_\varrho$ if $a_k<0$ (resp. $a_k>0$). 
\end{itemize}
\end{theorem}

\begin{proof}$(i)$ For every $k=1,\dots,n$, fix $R_k>0$  such that $B_{R_k}(x_k)\cap\{x_1,\dots,x_n\}=\{x_k\}$, where $B_R(x)$ is an open ball of 
radius $R$ centered at $x$. 
Now, define 
$u_{\varrho,k}(x):=u_\varrho(x+x_k)-u_\varrho(x_k)$ for every $x\in B_{R_k}(0)$. Since $\nabla u_{\varrho,k}(x)=\nabla u_\varrho(x+x_k)$ and $x\in B_{R_k}(x_k)\setminus\{x_k\}$ iff  $x-x_k\in B_{R_k}(0)\setminus\{0\}$, by Remark \ref{existence} we obtain that for every $\Omega\subset B_{R_k}(0)\setminus\{0\}$, $u_{\varrho,k}|_\Omega$ minimizes the functional
$I_\Omega: \mathcal C(u_{\varrho,k}|_{\partial\Omega},\bar\Omega)\to\mathbb R$  defined by 
$$I_\Omega(u):=\int_\Omega(1-\sqrt{1-|\nabla u|^2})dx.$$
 
Hence, the graph of $u_{\varrho,k}|_{B_{R_k}(0)}$ is an area maximizing hypersurface in the Minkowski space having an isolated singularity at 0, in the sense of \cite[Definitions 0.2 and 1.1]{Ecker}. By \cite[Theorem 1.5]{Ecker}, we can conclude that 0 is a light-cone-like singularity in the sense of \cite[Definition~1.4]{Ecker}. This implies that, for  every $x\in B_{R_k/t}(0)$ with $|x|=1$, 
$$
\lim_{h\to 0^+}\frac{u_{\varrho,k}(hx)}h\quad\mbox{exists and}\quad\left|\lim_{h\to 0^+}\frac{u_{\varrho,k}(hx)}h\right|=1.
$$
Since $u_{\varrho,k}(0)=0$, this means that for every direction $x$, there exists one sided directional derivative of $u_{\varrho,k}$ along $x$ at 0 and its absolute value is 1, that is,  
$$
\lim_{h\to 0^+}\frac{u_{\varrho,k}(hx+0)-u_{\varrho,k}(0)}h\quad\mbox{exists and}\quad\left|\lim_{h\to 0^+}\frac{u_{\varrho,k}(hx+0)-u_{\varrho,k}(0)}h\right|=1,
$$
which concludes the proof of ($i$). 

$(ii)$ Since 0 is a light-cone-like singularity of $u_{\varrho,k}|_{B_{R_k}(0)}$, two cases may occur (cf. \cite[Definition~1.4 and Lemma 1.9]{Ecker}): either 
$$u_{\varrho,k}>0\quad\mbox{in } B_{R}(0)\setminus\{0\}$$
or 
$$u_{\varrho,k}<0\quad\mbox{in } B_{R}(0)\setminus\{0\}$$
for some $0<R<R_k$. As a consequence, either $x_k$ is a relative strict minimizer of $u_\varrho$ or $x_k$ is a relative strict maximizer of $u_\varrho$.

Now, in order to detect which situation occurs depending on the sign of $a_k$, we use the comparison principle proved in Lemma \ref{comparison}. 
If $a_k<0$, we set $\Omega:=B_{R/2}(x_k)$, $\varrho_1:=0$, $\varphi_1:=0$, $\varrho_2:=a_k\delta_{x_k}$, and $\varphi_2:=u_{\varrho}|_{\partial B_{R/2}(x_k)}$. Hence, $u_1=0$, $u_2=u_\varrho|_{B_{R/2}(x_k)}$, and $\varrho_2\le\varrho_1$. Then, by Lemma \ref{comparison}
\begin{equation}\label{supsup}\sup_{B_{R/2}(x_k)} u_\varrho\le \sup_{\partial B_{R/2}(x_k)} u_\varrho.\end{equation}
Suppose by contradiction that $x_k$ is a relative strict maximizer of $u_\varrho$ in $B_R(x_k)$, then 
$$u_\varrho(x_k)=\sup_{B_{R/2}(x_k)} u_\varrho> \max_{\partial B_{R/2}(x_k)} u_\varrho,$$
which contradicts \eqref{supsup}. Thus, $x_k$ is a relative strict minimizer of $u_\varrho$.
Analogously, it is possible to prove that when $a_k>0$, $x_k$ is a relative strict maximizer of $u_\varrho$. 
\end{proof}

In what follows we give an explicit quantitative sufficient condition on the charge values $a_k$'s and on the charge positions $x_k$'s for $u_\varrho$ to be a classical solution of  
\begin{equation}\label{general}
-\mathrm{div}\left(\frac{\nabla u}{\sqrt{1-|\nabla u|^2}}\right)=0
\end{equation}
in some subset of $\mathbb R^N\setminus\{x_1,\dots,x_n\}$. As mentioned in the introduction, our results complement the qualitative ones contained in \cite{BDP} 
(see Theorem \ref{Le4.1} above), stating that if the charges are sufficiently small in absolute value or far away from each other, then the minimizer solves the problem.  

First, we prove the following lemma. 

\begin{lemma}\label{firstlemma}
Let $N\ge3$. There exists a constant $C=C(N)>0$ such that 
\begin{equation}\label{emb}
\|\nabla u\|^2_{L^2(\mathbb R^N)}\ge C\|u\|^N_{L^\infty(\mathbb R^N)},
\end{equation}
for all $u\in\mathcal X$. The best constant 
$$
\bar C:= \inf_{u\in\mathcal X\setminus\{0\}}\frac{\|\nabla u\|^2_{L^2(\mathbb R^N)}}{\|u\|^N_{L^\infty(\mathbb R^N)}}
$$
is achieved by a radial and radially decreasing function.
\end{lemma}

\begin{proof} For all $u\in\mathcal X\setminus\{0\}$, we define the ratio
$$
\mathcal R(u):=\frac{\|\nabla u\|^2_{L^2(\mathbb R^N)}}{\|u\|^N_{L^\infty(\mathbb R^N)}}
$$ 
and we observe that for any $t > 0$ it is 
%homogeneous of degree $2-N$ (i.e. $\mathcal R(\lambda u)=\lambda^{2-N}\mathcal R(u)$ for all $0<\lambda\le \|\nabla u\|^{-1}_{L^\infty(\mathbb R^N)}$) and 
invariant under the transformation $\phi_t : \mathcal X  \to \mathcal X$, with $\phi_t(v) :=  tv(\cdot/t)$ for all $v\in\mathcal X$. 

Furthermore, fix $u\in\mathcal X\setminus\{0\}$ and denote by $u^\star$ the symmetric decreasing rearrangement of $|u|$ (see e.g. \cite[Chapter 3]{liebanalysis}). Then, $\|u\|_{L^\infty(\mathbb R^N)}=\|u^\star\|_{L^\infty(\mathbb R^N)}$ and $\|\nabla u\|_{L^2(\mathbb R^N)}\ge\|\nabla u^\star\|_{L^2(\mathbb R^N)}$ by the Polya-Szeg\H{o} inequality. Hence, $\mathcal R(u)\ge\mathcal R(u^\star)$. Therefore, if we denote by $\mathcal X_-^{\mathrm{rad}}$ the set of $\mathcal X$-functions which are radial and radially decreasing, then
$$\bar C=\inf_{u\in\mathcal X\setminus\{0\}}\mathcal R(u)=\inf_{u\in\mathcal X_-^{\mathrm{rad}}\setminus\{0\}}\mathcal R(u).$$

Finally, we prove the existence of a minimizer of $\mathcal R$. Let $(u_n)\subset\mathcal X_-^{\mathrm{rad}}\setminus\{0\}$ be a minimizing sequence. Without loss of generality we may assume that $u_n(0) = \|u_n\|_{L^\infty(\mathbb R^N)}=1$ for all $n\in\mathbb N$, otherwise we transform it by an appropriate $\phi_t$. 
Then, 
$\|\nabla u_n\|^2_{L^2(\mathbb R^N)}\to\bar C$, and in particular $(u_n)$ is bounded in $\mathcal X$. Hence, up to a subsequence, $u_n \rightharpoonup \bar u$ in $\mathcal X$ and  $u_n \to\bar u$ uniformly on compact sets of $\mathbb R^N$, by Lemma \ref{lemma2.1}. In particular, $\bar u\in \mathcal X_-^{\mathrm{rad}}$, $1=u_n(0)\to\bar u(0)$, and so $\|\bar u\|_{L^\infty(\mathbb R^N)}=1$. Therefore, the weak lower semicontinuity of the norm yields 
$$
\mathcal R(\bar u)=\int_{\mathbb R^N}|\nabla \bar u|^2dx\le \liminf_{n\to\infty}\int_{\mathbb R^N}|\nabla u_n|^2dx=\inf_{u\in\mathcal X\setminus\{0\}}\mathcal R(u),
$$
and so $\bar u$ is a minimizer.   
\end{proof}

\begin{remark}
The exponent $N$ appearing in the right-hand side of \eqref{emb} naturally arises from the fact that $\mathcal R$ is invariant under transformations $\phi_t$.   
\end{remark}

\begin{lemma}\label{secondlemma}
The best constant for inequality \eqref{emb} is given by 
\begin{equation}\label{bestconstant}
\bar C=\frac{2}{N}\left(\frac{N-2}{N-1}\right)^{N-1}\omega_{N-1}.
\end{equation}
\end{lemma}
\begin{proof}
In order to find the explicit value of $\bar C$, we will build by hands a minimizer of $\mathcal R$. 

{\it Step 1: The minimizer can be found in a smaller function space.} We first observe that if $u\in\mathcal X$, then 
$\lambda u\in\mathcal X$ if and only if $0<\lambda\le\|\nabla u\|^{-1}_{L^\infty(\mathbb R^N)}$. Moreover, for all $\lambda\in (0,\|\nabla u\|^{-1}_{L^\infty(\mathbb R^N)}]$
$$
\mathcal R(\lambda u)=\lambda^{2-N}\mathcal R(u)\ge \frac{1}{\|\nabla u\|^{2-N}_{L^\infty(\mathbb R^N)}}\mathcal R(u)=\mathcal R\left(\frac{u}{\|\nabla u\|_{L^\infty(\mathbb R^N)}}\right). 
$$
Then, set 
$$
\widetilde{\mathcal X}:=\{u\in\mathcal X^{\mathrm{rad}}_-\,:\,u\ge0\mbox{ and }\mathrm{ess sup}\, |u'|=\mathrm{ess sup }\, u=1\},
$$
where with abuse of notation we have written $u(r):=u(x)$ for $r=|x|$.
Together with Lemma \ref{firstlemma}, we have 
$$
\bar C=\inf_{u\in\widetilde{\mathcal X}\setminus\{0\}}\mathcal R(u).
$$

{\it Step 2: The minimizer has non-decreasing first derivative.} Let $\bar u\in\widetilde X$ be any minimizer of $\mathcal R$ and consider any two (measurable) sets $S_1, S_2 \subset (0, \infty)$ of positive Lebesgue measure such that $\sup S_1<\inf S_2$. 
%Since $\bar u'$ is defined only up to set of measure zero, we will take $S_1, S_2$ with positive Lebesgue measure. 
For a contradiction assume that $\bar{u}' \leq B - \delta$ on $S_2$ and $0 \geq  \bar{u}'  \geq B + \delta$ on $S_1$ for some $B\in[-1,0)$ and $\delta\in(0,-B)$. 
Note that by making sets $S_1, S_2$ smaller 
if necessary (still of positive measure) we can assume that $\textrm{dist}(S_1, S_2) \geq \varepsilon$ and $S_1 \cup S_2$ is bounded.  
Since $S_1$ and $S_2$ have positive measure, it is standard to see 
that there exists a translation of $S_1$, denoted by $S_1 + k$ for some $k \geq \varepsilon$, such that $M_2 := (S_1 + k) \cap S_2$ has positive 
measure.
Denote $M_1 := M_2 - k$ and note that $M_1 \subset S_1$. 
Of course $M_1$ and $M_2$ are measurable, with positive measure. 

Define a new function 
\begin{equation*}
w'(r) :=
\begin{cases}
\bar{u}'(r + k) & r \in M_1 \\
\bar{u}'(r - k) & r \in M_2 \\
 \bar{u}'(r) & \textrm{otherwise} \,,
\end{cases}
\end{equation*}
that is, we exchange the values of $\bar{u}'$ on sets $M_1$ and $M_2$. Note that $w' \in L^2((0, \infty))$ and it is the derivative of the 
function $w(r) = 1 + \int_0^r w'(s) \, ds$, which is decreasing by Lemma \ref{firstlemma}, belongs to $L^2((0, \infty))$, and has $w (0) = 1$. Observe that $w \equiv \bar{u}$ 
outside of the convex hull of $S_1 \cup S_2$.  
Then, 
$$
\begin{aligned}
\|\nabla &\bar{u}\|_{L^2(\mathbb{R}^N)}^2 -  \|\nabla w \|_{L^2(\mathbb{R}^N)}^2 =
\int_0^\infty |\bar{u}'|^2 r^{N - 1} \, dr - \int_0^\infty |w'|^2 r^{N - 1} \, dr \\
& = \int_{M_1} ( |\bar{u}'(r)|^2 -  |\bar{u}'(r + k)|^2) r^{N - 1} \, dr + 
\int_{M_2} ( |\bar{u}'(r)|^2 -  |\bar{u}'(r - k))|^2) r^{N - 1} \, dr \\
 &= \int_{M_1} (|\bar{u}'(r + k)|^2 -  |\bar{u}'(r))|^2) [(r + k)^{N - 1} - r^{N - 1}] \, dr > 0 \,,
\end{aligned}
$$
a contradiction to $\bar{u}$ being a minimizer. Note that we used that for $r \in M_1$ one has $r + k \in M_2$, and consequently since $B < 0$,
$|\bar{u}'(r + k)|^2 \geq (B - \delta)^2  >  (B + \delta)^2 \geq |\bar{u}'(r)|^2$. Moreover, 
 $k \geq \varepsilon > 0$ and the strict inequality follows. 
By the arbitrariness of $0<\delta<-B$, we obtain that $\bar{u}'$ is a non-decreasing function.

{\it Step 3: The minimizer is harmonic outside the set of points of -1 derivative.} Denote $R := \sup \{r \in (0, 1) : \bar{u}'(r) = -1 \}$ and set $R = 0$ if $\bar{u}'(r) > -1$ for each $r > 0$.  Fix any $\varepsilon > 0$
and note that $B := \bar{u}'(R + \varepsilon) > -1$. Therefore, $\bar{u}'(r) \geq B > -1$ on $(R + \varepsilon, \infty)$.  

In order to prove that at points $r$ where $\bar u'(r)\neq -1$, $\bar u$ is harmonic, fix any smooth  $\psi\in C^1_{\mathrm{c}}((R + \varepsilon, \infty))$ and note that for sufficiently small (in absolute value)
$\xi$, one has $(\bar{u} + \xi \psi)' \geq -1$. Then, by the minimality of $\bar{u}$,
\begin{equation*}
0 \geq \int_0^\infty |\bar{u}'|^2 r^{N - 1} \, dr - \int_0^\infty |\bar{u}' + \xi \psi'|^2 r^{N - 1} \, dr = - \xi \int_0^\infty (2\bar{u}' \psi ' + \xi |\psi'|^2) r^{N - 1} \, dr  \,.
\end{equation*}  
Since $|\xi|\ll 1$ is arbitrarily small, positive or negative, we obtain 
\begin{equation*}
0 =  \int_0^\infty \bar{u}' \psi '  r^{N - 1} \, dr = - \int_0^\infty (\bar{u}'   r^{N - 1})' \psi  \, dr \,.
\end{equation*}
By the arbitrariness of $\psi$, this implies that $ (\bar{u}'   r^{N - 1})' = 0$ a.e. in $(R + \varepsilon, \infty)$, which in turn gives that $\bar{u}$ is harmonic in $(R,\infty)$, because $\varepsilon > 0$ is arbitrary.

{\it Step 4: The explicit form of a minimizer.} Altogether, we have proved that a minimizer $\bar u$ of $\mathcal R$ can be taken of the form 
$$
\bar u(r)=
\begin{cases}
1-r\quad&\mbox{if }r\in(0,R),\\
c_1r^{2-N}+c_2&\mbox{if }r\in[R,\infty)
\end{cases}
$$
for suitable constants $c_1,\,c_2>0$ and $R\ge0$.  Since $\lim_{r\to\infty}\bar u(r)=0$, $c_2=0$ and 
since $r^{2 - N}$ is unbounded at $0$, we have $R > 0$ and clearly $R \leq 1$. Moreover, $\bar u$ is continuous and $|\bar u'|\le1$, that is
$$
c_1=R^{N-2}(1-R)\quad\mbox{and}\quad c_1\le\frac{R^{N-1}}{N-2}.
$$   
Consequently, $R \geq \frac{N - 2}{N - 1}$. 
Now, we minimize  $\|\nabla \bar u\|_{L^2(\mathbb R^N)}^2$ as a function of $R$, or equivalently we minimize
$$\begin{aligned}E(R)&:=\int_0^{+\infty}\bar u'^2(r)r^{N-1}dr=\int_0^Rr^{N-1}dr+\int_R^{+\infty}c_1^2(N-2)^2r^{1-N}dr \,.
\end{aligned}$$
Using the bound on $c_1$ we have 
\begin{equation}
E'(R) = R^{N-1} - c_1^2(N - 2)^2 R^{1 - N} \geq 0 \,,
\end{equation}
and therefore $E$ is a non-decreasing function. Thus, the minimum is attained at $\bar R:=\frac{N-2}{N-1}$ and since 
$\bar C= E(\bar R)\omega_{N-1}$, we obtain the desired assertion. 
\end{proof}

We are now ready to prove the Theorem \ref{prop>2charges}. 
Let $\varrho=\sum_{k=1}^n a_k\delta_{x_k}$ and
$$
\begin{gathered}
\mathcal K_+:=\{k\in\mathbb N\,:\, 1\le k\le n\mbox{ and } a_k>0\},\\
\mathcal K_-:=\{k\in\mathbb N\,:\, 1\le k\le n\mbox{ and } a_k<0\}.
\end{gathered}
$$

\begin{proof}[$\bullet$ Proof of Theorem \ref{prop>2charges}]
Without loss of generality assume $j\in\mathcal K_+$ and $l\in\mathcal K_-$. Let $u_{\pm}\in\mathcal X\setminus\{0\}$ be the unique minimizers of 
$$
I_{\pm}(u):=\int_{\mathbb R^N}(1-\sqrt{1-|\nabla u|^2})dx-\sum_{k\in \mathcal K_{\pm}}a_k u(x_k),
$$
respectively. By Proposition \ref{Prop2.3}
\begin{equation}\label{0<Ipm}
0>I_{\pm}(u_\pm)\ge \frac{1}{2}\|\nabla u_{\pm}\|^2_{L^2(\mathbb R^N)}-\left(\sum_{k\in\mathcal K_\pm}|a_k|\right)\|u_\pm\|_{L^\infty(\mathbb R^N)},
\end{equation}
where we have used the inequality $\frac{1}{2}t\le 1-\sqrt{1-t}$ for $t\in[0,1]$.  
On the other hand, by Lemma \ref{firstlemma}, we have 
$$
\|\nabla u_\pm\|^2_{L^2(\mathbb R^N)}\ge \bar C \|u_\pm\|^N_{L^\infty(\mathbb R^N)}.
$$
Together with \eqref{0<Ipm} this gives
\begin{equation}\label{roughpm}
\|u_\pm\|_{L^\infty(\mathbb R^N)}\le\left(\frac{2}{\bar C}\sum_{k\in\mathcal K_\pm}|a_k|\right)^{\frac1{N-1}} 
\end{equation}
and in particular 
\begin{equation}\label{estimateupmxk}
\pm u_\pm(x_j)=|u_\pm(x_j)|\le \left(\frac{2}{\bar C}\sum_{k\in\mathcal K_\pm}|a_k|\right)^{\frac1{N-1}}\quad\mbox{for all }j\in\{1,\dots n\}, 
\end{equation}
since $u_+\ge 0$ and $u_-\le 0$ in all of $\mathbb R^N$, by the Comparison Principle \cite[Lemma~2.12]{BDP}.
By the same principle, we also know that 
$$
u_-(x)\le u_{\varrho}(x)\le u_+(x)\qquad\mbox{for all }x\in\mathbb R^N.
$$
Hence, by \eqref{estimateupmxk}, \eqref{hpjl}, and \eqref{bestconstant}
\begin{equation}\label{uajl}
\begin{aligned}
u_\varrho(x_j)-u_\varrho(x_l)&\le u_+(x_j)- u_-(x_l)\\
&\le \left(\frac{2}{\bar C}\sum_{k\in\mathcal K_+}|a_k|\right)^{\frac1{N-1}}+\left(\frac{2}{\bar C}\sum_{k\in\mathcal K_-}|a_k|\right)^{\frac1{N-1}}\\
&< \min_{\underset{h\neq i}{h,i\in\{1,\dots,n\}}}|x_h - x_i|\le |x_j-x_l|.
\end{aligned}
\end{equation} 
By Theorem \ref{Le4.1} either $u_\varrho$ is smooth on $\mathrm{Int}(\overline{x_jx_l})$, or  
\begin{equation}\label{segmentjl}
u_\varrho(tx_l+(1-t)x_j)=t u_\varrho(x_l)+(1-t)u_\varrho(x_j)\quad\mbox{for all }t\in(0,1).
\end{equation}
For a contradiction assume \eqref{segmentjl}. Then,  Theorem~\ref{maxmin} yields that $x_j$ is a strict relative maximizer and
$$
\lim_{t\to0^+}\frac{u_\varrho(t(x_l-x_j)+x_j)-u_\varrho(x_j)}{t|x_l-x_j|}=-1.
$$
By \eqref{segmentjl}, this gives immediately 
\begin{equation}\label{onthesegmentjl}
\frac{u_\varrho(x_l)-u_\varrho(x_j)}{|x_l-x_j|}=-1.
\end{equation}
Whence, together with \eqref{uajl}, we have %$u_\varrho(x_l) < u_\varrho(x_j)$. By \eqref{hpjl} and \eqref{uajl}, we obtain
$$
-|x_l-x_j|<u_\varrho(x_l)-u_\varrho(x_j)=-|x_l-x_j|,
$$
a contradiction. 
We can now repeat the same argument for all the couples of point charges and conclude the proof.
\end{proof}

\begin{remark}\label{3.8} By \eqref{uajl} it is apparent that under the weaker assumption
$$
\left(\frac{N}{\omega_{N-1}}\right)^{\frac{1}{N-1}}\frac{N-1}{N-2}\left[\left(\sum_{k\in \mathcal K_+}a_k\right)^{\frac{1}{N-1}}+\left(\sum_{k\in \mathcal K_-}|a_k|\right)^{\frac{1}{N-1}}\right] <|x_j-x_l|,
$$
we get the result (i.e., $u_\varrho$ is a classical solution) only along the line segment $\mathrm{Int}(x_jx_l)$.
 
Furthermore, it is possible to refine \eqref{roughpm}, and consequently the sufficient condition \eqref{hpjl}, by replacing \eqref{emb} with the following inequality 
\begin{equation}\label{refinedemb}
\int_{\mathbb R^N} \left(1-\sqrt{1-|\nabla u|^2}\right)dx\ge \widetilde C\|u\|^N_{L^\infty(\mathbb R^N)}\quad\mbox{for all }u\in\mathcal X
\end{equation}
and for some $\widetilde C=\widetilde C(N)\ge \frac{\bar C}2$. Indeed, suppose we have already proved \eqref{refinedemb}. Starting as in the proof of Theorem \ref{prop>2charges}, we have
$$
0>I_{\pm}(u_\pm)\ge \int_{\mathbb R^N}\left(1-\sqrt{1-\|\nabla u_{\pm}\|^2}\right) dx-\left(\sum_{k\in\mathcal K_\pm}|a_k|\right)\|u_\pm\|_{L^\infty(\mathbb R^N)}
$$
that, combined with \eqref{refinedemb}, gives
$$
\|u_\pm\|_{L^\infty(\mathbb R^N)}\le\left(\frac{1}{\widetilde C}\sum_{k\in\mathcal K_\pm} |a_k|\right)^{\frac{1}{N-1}}.
$$
Hence, it is enough to require 
\begin{equation}\label{moreprecise}
\widetilde C^{-\frac{1}{N-1}}\left[\left(\sum_{k\in\mathcal K_+}a_k\right)^{\frac{1}{N-1}}+\left(\sum_{k\in\mathcal K_-} |a_k|\right)^{\frac{1}{N-1}}\right] 
< |x_j-x_l|
\end{equation}
(which is a weaker assumption than \eqref{hpjl}, since $\widetilde C^{-\frac1{N-1}}\le (\bar C/2)^{-\frac1{N-1}}$) to conclude the statement of Theorem \ref{prop>2charges}.
As in Lemma \ref{firstlemma} (see also \cite{bonheure2012infinitely}) we can show that 
$\widetilde C$ is attained by the unique weak solution $\tilde u$ of the problem 
$$
\begin{cases}-\mathrm{div}\left(\frac{\nabla u}{\sqrt{1-|\nabla u|^2}}\right)=a\delta_0\quad\mbox{in }\mathbb R^N,\\
\lim_{|x|\to\infty}u(x)=0
\end{cases}
$$
with $a:=A(N)^{1-N}$ and 
\begin{equation}\label{AN}
A(N):=\omega_{N-1}^{-\frac1{N-1}}\int_0^{+\infty}\frac{ds}{\sqrt{s^{2(N-1)}+1}},
\end{equation}
cf. \cite[Theorem 1.4]{BDP}. 
Such $\tilde{u}$ is radial and radially decreasing, and the previous problem in radial coordinates reads as 
$$
\begin{cases}
\left(r^{N-1}\frac{u'}{\sqrt{1-(u')^2}}\right)'=0\quad\mbox{in }(0,\infty),\\
u(0)=1, \quad\lim_{r\to\infty}u(r)=0,
\end{cases}
$$
where as usual we have written $u(r):=u(x)$ for $r=|x|$. Therefore, 
$$
\tilde u(r)=\int_r^{+\infty}\frac{a/\omega_{N-1}}{\sqrt{s^{2(N-1)}+(a/\omega_{N-1})^2}}ds, 
$$
see below for a similar argument. Hence,
\begin{equation}\label{Ctilde}
\begin{aligned}
\widetilde C&=\omega_{N-1}\int_0^\infty r^{N-1}\left(1-\sqrt{1-(\tilde u'(r))^2}\right)dr\\
&=\omega_{N-1}\frac{\displaystyle \int_0^\infty r^{N-1}\left(1-\frac{r^{N-1}}{\sqrt{r^{2(N-1)}+1}}\right)dr}{\left(\displaystyle \int_0^\infty\frac1{\sqrt{r^{2(N-1)}+1}}dr\right)^N}.
\end{aligned}
\end{equation}
We can numerically check that, for example when $N=3$, 
$$\bar C=\frac{\omega_2}{6}\le 2\widetilde C\,\approx 2\cdot 0,097\,\omega_2.$$
\end{remark}

To end this section, we consider the case of two point charges of different sign, namely
\begin{equation}\label{a12}
\varrho:=a_1\delta_{x_1}+a_2\delta_{x_2},
\end{equation}
with $a_1\cdot a_2<0$. In this case, we can give a more precise sufficient condition. %\alert{is there any necessary condition? Is it believed to hold all the time?}.

\begin{proposition}\label{2points} Let $\varrho$ be as in \eqref{a12}. If $a_1\cdot a_2 < 0$ and 
$$
\left(|a_1|^{\frac1{N-1}}+|a_2|^{\frac1{N-1}}\right)A(N) < |x_1-x_2|,
$$
where $A(N)$ is defined in \eqref{AN}, then $u_\varrho\in C^\infty(\mathbb R^N\setminus\{x_1,x_2\})\cap C(\mathbb R^N)$, it is a classical solution of \eqref{general} and it is strictly spacelike in $\mathbb R^N\setminus\{x_1,x_2\}$.
\end{proposition}

\begin{proof}
It is standard to prove that for $k=1,\,2$ the unique solution $\tilde u_k$ of 
\begin{equation}\label{fond-k}
-\mathrm{div}\left(\frac{\nabla u}{\sqrt{1-|\nabla u|^2}}\right)=a_k\delta_{x_k}\qquad\mbox{in }\mathbb R^N,
\end{equation}
with $\lim_{|x|\to\infty}u=0$, is radial about $x_k$ and satisfies 
\begin{equation}\label{radial}
\frac{r^{N-1}\tilde u_k'(r)}{\sqrt{1-\tilde u_k'(r)^2}}=C\quad\mbox{in }\mathbb R^N\setminus\{x_k\}\mbox{ for some $C\in\mathbb R$,}
\end{equation}
where with abuse of notation  $\tilde u_k(r)=\tilde u_k(|x-x_k|)$ and $'$ denotes the derivation with respect to $r:=|x-x_k|$. In particular, by \eqref{radial}, 
 $\tilde u'_k$ never changes sign, and therefore $\tilde u_k$ is monotone in $r$.
Since $\tilde u_k$ vanishes at infinity, by \eqref{radial}  
we obtain
$$
\begin{aligned}
-C\tilde u_k(0)&=C\left(\lim_{r\to+\infty}\tilde u_k(r)-\tilde u_k(0)\right)=C\int_0^{+\infty}\tilde u'_k(r)dr\\
&=\int_0^{+\infty}\frac{r^{N-1}\tilde u_k'(r)^2}{\sqrt{1-\tilde u_k'(r)^2}}dr=\frac{a_k}{\omega_{N-1}} \tilde u_k(0).
\end{aligned}
$$
Since $\tilde u_k$ is monotone in $r$, $a_k\neq0$, and $\lim_{r\to\infty}\tilde u_k=0$, we have that $\tilde u_k(0)\neq0$, whence $C=-a_k/\omega_{N-1}$.
Furthermore, by solving for $\tilde u_k'$ in \eqref{radial} and integrating we have 
$$
\tilde u_k(r)=\int_r^{+\infty}\frac{a_k/\omega_{N-1}}{\sqrt{s^{2(N-1)}+(a_k/\omega_{N-1})^2}}ds\qquad\mbox{for }k=1,\,2, 
$$
and in particular 
\begin{equation}\label{fondsolk}
\tilde u_k(0)=\mathrm{sign}(a_k)|a_k|^{\frac{1}{N-1}}A(N)\qquad\mbox{for }k=1,\,2.
\end{equation}

Since $a_1>0>a_2$, $a_2\delta_{x_2}\le \varrho\le a_1\delta_{x_1}$ (cf. \cite[Definition~2.11]{BDP}). By the Comparison Lemma 2.12 of \cite{BDP}, we know that 
$$
\tilde u_2(x)\le u_{\varrho}(x)\le \tilde u_1(x)\qquad\mbox{for all }x\in\mathbb R^N.
$$
The conclusion follows exactly as in Theorem \ref{prop>2charges}. 
%Analogously,  
%$$
%u_{\varrho}(x)\ge \tilde u_2(x)\qquad\mbox{for all }x\in\mathbb R^N.
%$$
%In particular, by \eqref{fondsolk}
%\begin{equation}\label{ua12}
%u_\varrho(x_1)-u_\varrho(x_2)\le \tilde u_1(x_1)-\tilde u_2(x_2)=\left(a_1^{\frac1{N-1}}+|a_2|^{\frac1{N-1}}\right)A(N).
%\end{equation} 
%
%Let $\overline{x_1x_2}$ denote the line segment joining $x_1$ and $x_2$. Now, by \cite[Lemma 4.1-$(iii)$]{BDP} two cases occur: either $u_\varrho$ is a classical (strictly spacelike) solution on $\overline{x_1x_2}$, or  
%\begin{equation}\label{segment}
%u_\varrho(tx_2+(1-t)x_1)=t u_\varrho(x_2)+(1-t)u_\varrho(x_1)\quad\mbox{for all }t\in(0,1).
%\end{equation}
%In the first case there is nothing to prove. In the second case, since by Theorem~\ref{maxmin}-$(ii)$ we know that $x_1$ is a strict relative maximizer, we have by Theorem \ref{maxmin}-$(i)$ that
%$$\lim_{t\to0^+}\frac{u_\varrho(t(x_2-x_1)+x_1)-u_\varrho(x_1)}{t|x_2-x_1|}=-1.$$
%By \eqref{segment}, this yields immediately 
%\begin{equation}\label{onthesegment}
%\frac{u_\varrho(x_2)-u_\varrho(x_1)}{|x_2-x_1|}=-1.
%\end{equation}
%Hence, $u_\varrho(x_1)>u_\varrho(x_2)$. Since by assumption
%$$
%\left(a_1^{\frac1{N-1}}+|a_2|^{\frac1{N-1}}\right)A(N)\le (1-\varepsilon)|x_1-x_2|
%$$
%for some $\varepsilon\in(0,1)$, \eqref{ua12} gives
%$$
%u_\varrho(x_1)-u_\varrho(x_2)<|x_1-x_2|,
%$$
%in contradiction with \eqref{onthesegment}. Together with \cite[Lemma 4.1]{BDP}, this concludes the proof. 
\end{proof}

\section{Approximating problem}\label{Sec3}
In this section we study some qualitative properties of the approximating solutions $u_m$ of the problem 
\eqref{Pa}. In particular, we focus on the regularity of $m_m$ in Proposition \ref{reg} and on their local behavior near the singularities $x_k$'s, proving Theorem \ref{approx_nabla}. From these results, it is apparent that $u_m$'s behavior resembles the behavior of the minimizer $u_\varrho$ that we approximate, see also the introduction for more comments.
 
\begin{proof}[$\bullet$ Proof of Proposition \ref{reg}] Let us denote 
$$\begin{aligned}&A(p):=\sum_{h=1}^m {\alpha_h}|p|^{2h-2}p,\\ 
&a^{ij}(p):=\frac{\partial A_i}{\partial p_j}=\sum_{h=1}^m {\alpha_h}\left[(2h-2)|p|^{2h-4}p_ip_j+|p|^{2h-2}\delta_{ij}\right],\\
&F(t):=\sum_{h=1}^m{\alpha_h}t^{2h-2}
\end{aligned}$$
for every $p\in\mathbb R^N$ and $t\ge0$, where $\delta_{ij}$ is the Kronecker delta. 
Then, by straightforward calculations we have for all $p,\,\xi\in\mathbb R^N$ 
$$
\begin{aligned}
&\sum_{i,j=1}^Na^{ij}(p)\xi_i\xi_j=\left(\sum_{h=1}^m\alpha_h|p|^{2h-2}\right)|\xi|^2+(p\cdot\xi)^2\sum_{h=1}^m\alpha_h(2h-2)|p|^{2h-4}\ge F(|p|)|\xi|^2,\\
&|a^{ij}(p)|\le \sum_{h=1}^m\alpha_h|p|^{2h-2}+\sum_{h=1}^m  \alpha_h(2h-2)|p|^{2h-2}\le (2m-1)F(|p|),\\
&|A(p)|=  \sum_{h=1}^m\alpha_h|p|^{2h-2} |p| =|p|F(|p|).
\end{aligned}
$$
Therefore, the operator $-\sum_{h=1}^m\alpha_h\Delta_{2h}$ and the function $F$ satisfy the hypotheses of \cite[Lemma~1]{Lieberman} with 
$\Lambda=(2m-1)$. 
To verify the last assumption in \cite[Lemma~1]{Lieberman},  let $u_m$ be the solution of \eqref{Pa}. Since $2m>\max\{N,2^*\}$,  
one has $\mathcal X_{2m} \hookrightarrow C^{0,\beta_m}_0(\mathbb R^N)$ and in particular, $u_m \in \mathcal X_{2m}$ is bounded. 
Let $B_{4R}$ be any ball of radius $4R$, such that $x_k\not\in B_{4R}$ for any $k=1, \cdots, n$. Then  $u_m$ satisfies 
$$
-\mathrm{div}\left(\sum_{h=1}^m{\alpha_h}|\nabla u_m|^{2h-2}\nabla u_m\right)=0\quad	\mbox{in }B_{4R}\quad\mbox{in the weak sense,}
$$
and since $u_m \in \mathcal X_{2m}$,
$$
\int_{B_{4R}}F(|\nabla u_m|)(1+|\nabla u_m|)^2dx<\infty.
$$
Therefore, by \cite[Lemma 1]{Lieberman}, $u_m\in C^{1,\beta}(B_R)$ for some $\beta\in(0,1)$, and $B_R$ has the same center as $B_{4R}$. We consider now the linear Dirichlet problem
\begin{equation}\label{L}
\begin{cases}
L_m u:=-\mathrm{div}\left(\displaystyle{\sum_{h=1}^m}\alpha_h|\nabla u_m|^{2h-2}\nabla u\right)=0\quad&\mbox{in } B_R,\\
u=u_m&\mbox{on }\partial B_R.
\end{cases}
\end{equation}
Clearly, $u_m$ is a weak solution of \eqref{L}. The boundary datum $u_m$ is continuous on $\partial B_R$, the operator $L_m$ is strictly elliptic in $B_R$ and has coefficients in $C^{0,\beta}(B_R)$. Hence, by \cite[Theorem 6.13]{GT}, \eqref{L} has a unique solution in $C(\bar{B}_R)\cap C^{2,\beta}(B_R)$, whence  $u_m\in C(\bar{B}_R)\cap C^{2,\beta}(B_R)$. We consider again \eqref{L}. Now we know that the coefficients of $L_m$ are of class $C^{1,\beta}(B_R)$ and that $u_m$ is a $C^2$-solution of the equation in \eqref{L}. By \cite[Theorem 6.17]{GT}, $u_m\in C^{3,\beta}(B_R)$. By a bootstrap argument, we obtain $u_m\in C^\infty(B_R)$. By the arbitrariness of $R$ and of the center of the ball $B_R$, $u_m\in C^\infty(\mathbb R^N\setminus\{x_1,\dots,x_n\})$.
\end{proof}

\begin{remark}\label{roleoflapl}
The presence of the Laplacian in the operators sum $\sum_{h=1}^m\alpha_h\Delta_{2h}$ plays an essential role in the proof of the previous result. Indeed, we observe that, among the hypotheses on $F$, \cite[Lemma~1]{Lieberman} requires $F(t)\ge\varepsilon > 0$ for all $t\ge0$, which is satisfied with $\varepsilon=\alpha_1$ thanks to the presence of the Laplacian.
\end{remark}

Next, we study the behavior of the solution $u_m$ of \eqref{Pa} and of its gradient, near the point charges $x_k$'s. 

\begin{proof}[$\bullet$ Proof of Theorem \ref{approx_nabla}] For any $k=1,\dots,n$, fix $R_k>0$ so small that $B_{R_k}(x_k)\cap\{x_1,\dots,x_n\}=\{x_k\}$. Then, $u_m$ solves
\begin{equation}\label{Pmk}
\begin{cases}
-\sum_{h=1}^m\alpha_h\Delta_{2h}u=a_k\delta_{x_k}\quad&\mbox{in }B_{R_k}(x_k),\\
u=u_m&\mbox{on }\partial B_{R_k}(x_k)
\end{cases}
\end{equation}
for all $k=1,\dots,n$. We split the proof into six steps. \smallskip

{\it Step 1: Translation.} For all $\varphi\in C_{\mathrm{c}}^\infty(B_{R_k}(x_k))$
\begin{equation}\label{wsm}
\sum_{h=1}^m\int_{B_{R_k}(x_k)}\alpha_h|\nabla u_m|^{2h-2}\nabla u_m\cdot\nabla \varphi dx=a_k\varphi(x_k).
\end{equation}
So, if we define $u_{m,k}(x):=u_m(x+x_k)-u_m(x_k)$ and $\varphi_k(x):=\varphi(x+x_k)$ for all $x\in B_{R_k}(0)$, we get $u_{m,k}\in C^\infty(B_{R_k}(0)\setminus\{0\})$, $\varphi_k\in C^\infty_\mathrm{c}(B_{R_k}(0))$ and
\begin{equation}\label{star}
\sum_{h=1}^m\int_{B_{R_k}(0)}\alpha_h|\nabla u_{m,k}|^{2h-2}\nabla u_{m,k}\cdot\nabla \varphi_k dx=a_k\varphi_k(0).
\end{equation}
Hence, by the arbitrariness of $\varphi\in C_{\mathrm{c}}^\infty(B_{R_k}(x_k))$, $u_{m,k}$ solves weakly 
\begin{equation}\label{Pmk0}
\begin{cases}
-\sum_{h=1}^m\alpha_h\Delta_{2h}u=a_k\delta_{0}\quad&\mbox{in }B_{R_k}(0),\\
u=u_{m,k}&\mbox{on }\partial B_{R_k}(0).
\end{cases}
\end{equation}  
Of course we have $u_{m, k}(0) = 0$.
\smallskip

{\it Step 2: Potential estimates on $u_{m,k}$.} Consider the operator 
$$-\sum_{h=1}^m\alpha_h\Delta_{2h}u=-\mathrm{div}\left(\frac{g(|\nabla u|)}{|\nabla u|}\nabla u\right),$$
with $g(t):=\sum_{h=1}^m\alpha_h t^{2h-1}$ for all $t\ge 0$, and note that 
$$1\le \frac{g'(t)t}{g(t)}\le 2m-1\quad\mbox{for all }t>0.$$
By \cite[Theorem 1.2]{Baroni}, for every $x_0\in B_{R_k}(0)$ Lebesgue point of $\nabla u_{m,k}$ and for every ball $B_{2R}(x_0)\subset B_{R_k}(0)$
one has
\begin{equation}\label{potentialest}
g(|\nabla u_{m,k}(x_0)|)\le c \mathbb I_1^{|a_k\delta_0|}(x_0,2R)+c g\left(\Mint_{B_R(x_0)}|\nabla u_{m,k}|dx\right) \,,
\end{equation}
where $c=c(N,m)>0$ and
$$\mathbb I_1^{|a_k\delta_0|}(x_0,R):=\int_0^R\frac{|a_k\delta_0|(B_\rho(x_0))}{\rho^N}d\rho$$
is the truncated linear Riesz potential of the measure $|a_k\delta_0|$.
Now, if $0 < |x_0| < R_k - 2R$
\begin{equation}\label{I1}
\mathbb I_1^{|a_k\delta_0|}(x_0,2R)=\int_{|x_0|}^{2R}\frac{|a_k|}{\rho^N}d\rho\le \frac{|a_k|}{(N-1)|x_0|^{N-1}}.
\end{equation}
If furthermore $R>R_k/4$ it follows for almost every $x_0$ that 
\begin{equation}\label{gmedia}
\begin{aligned}g\left(\Mint_{B_R(x_0)}|\nabla u_{m,k}|dx\right)&< \sum_{h=1}^m\alpha_h\left(\frac{\|\nabla u_{m,k}\|_{L^1(B_R(x_0))}}{|B_{R_k/4}|}\right)^{2h-1}\\
&\le\sum_{h=1}^m\alpha_h\left(\frac{\|\nabla u_{m,k}\|_{L^1(B_{R_k}(0))}}{|B_{R_k/4}|}\right)^{2h-1}=: C,\end{aligned}
\end{equation}
where $C=C\left(\|\nabla u_m\|_{L^1(B_{R_k}(x_k))}, N,g\right)>0$ is independent of the specific $x_0$ and $R$ considered. We note that, if $|x_0|<R_k/4$, then \eqref{potentialest}-\eqref{gmedia} hold with any $R\in(R_k/4,3R_k/8)$.
Therefore, by combining \eqref{gmedia}  with \eqref{potentialest} and \eqref{I1}, we obtain for a.e. $x\in B_{R_k/4}(0)$ 
\begin{equation}\label{est_for_nabla_umk}
\begin{aligned}
|\nabla u_{m,k}(x)|&= \left(\frac{g(|\nabla u_{m,k}(x)|)}{\alpha_m}\right)^{\frac1{2m-1}}\\
&\le  \left\{\frac{c}{\alpha_m|x|^{N-1}}\left[\frac{|a_k|}{N-1}+C\left(\frac{R_k}{4}\right)^{N-1}\right]\right\}^{\frac1{2m-1}}=: \frac{C'}{|x|^{\frac{N-1}{2m-1}}},
\end{aligned}
\end{equation}
with $C'=C'(\|\nabla u_m\|_{L^1(B_{R_k}(x_k))},R_k,|a_k|,N,m,g)>0$.
\smallskip

{\it Step 3: Scaling.} Fix two integers $m>\max\{N/2,2^*/2\}$ and $k\in\{1,\dots,n\}$. For any $\varepsilon>0$ and $x\in B_{R_k/\varepsilon}(0)\setminus\{0\}$, define 
$u_\varepsilon(x):=\varepsilon^{\frac{N-2m}{2m-1}}u_{m,k}(\varepsilon x)$. 
Then $u_\varepsilon\in C^\infty(B_{R_k/\varepsilon}(0)\setminus\{0\})$, and $\nabla u_\varepsilon(x) = \varepsilon^{\frac{N-1}{2m-1}}\nabla u_{m,k}(\varepsilon x)$  for all $x\in B_{R_k/\varepsilon}(0)\setminus\{0\}$. By substituting into \eqref{star}, we obtain for any $\varphi\in C_\mathrm{c}^\infty(B_{R_k/\varepsilon}(0))$
$$
\sum_{h=1}^m\int_{B_{R_{k}/\varepsilon}(0)}\varepsilon^{N-2h+\frac{(2m-N)(2h-1)}{2m-1}}\alpha_h|\nabla u_\varepsilon|^{2h-2}\nabla u_\varepsilon\cdot\nabla \varphi dx=a_k\varphi(0),
$$
or in other words $u_\varepsilon$ solves weakly 
\begin{equation}\label{Eeps}
-\sum_{h=1}^m \varepsilon^{N-2h+\frac{(2m-N)(2h-1)}{2m-1}}\alpha_h\Delta_{2h}u=a_k\delta_{0}\quad\mbox{in }B_{R_k/\varepsilon}(0).
\end{equation}  
We note that the exponent of $\varepsilon$ is positive for $h<m$ and is zero for $h=m$. Also note that 
$u_\varepsilon(0) = 0$. \smallskip

{\it Step 4: Limit as $\varepsilon\to0$.} 
In terms of $u_\varepsilon$,  \eqref{est_for_nabla_umk} translates for a.e. $x\in B_{R_k/4\varepsilon}(0)$ to a global estimate
\begin{equation}\label{est_nabla-u-eps}
|\nabla u_\varepsilon(x)|\le C'|x|^{\frac{1-N}{2m-1}}.
\end{equation}
Since $2m > N$,  for fixed $\bar R\in(0,R_k/4\varepsilon)$,  \eqref{est_nabla-u-eps} yields 
\begin{equation}\label{boundedness}
\int_{B_{\bar R}(0)}|\nabla u_\varepsilon|^{2m}dx\le \frac{2m-1}{2m-N}C'^{2m}\omega_{N-1}\bar R^{\frac{2m-N}{2m-1}}=:C'',
\end{equation}
where $C''=C''(\|\nabla u_m\|_{L^1(B_{R_k}(x_k))},|a_k|,N,m,g,\bar R)>0$ independent of $\varepsilon$. 

Next, we obtain local estimates uniform in $\varepsilon$. 
Let $A \subset B_{\bar R}(0)\setminus\{0\}$ be a compact set. 
Then, by \eqref{est_nabla-u-eps} and since $u_\varepsilon(0) = 0$,
\begin{equation}\label{ubd}
|u_\varepsilon(x)|\le \int_0^1|\nabla u_\varepsilon(tx)| |x|dt\le C'\frac{2m-1}{2m-N}\bar R^{\frac{2m-N}{2m-1}}\quad\mbox{for all }x\in A.
\end{equation}
Furthermore, by Proposition~\ref{reg} we have
$$
\begin{aligned}
|\nabla u_\varepsilon(x)-\nabla u_\varepsilon(y)|
&=\varepsilon^{\frac{N-1}{2m-1}} |\nabla u_{m,k}(\varepsilon x)-\nabla u_{m,k}(\varepsilon y)|
\\
&\le \varepsilon^{\frac{N-1}{2m-1}+1 - \beta_m}|x-y|^{1 - \beta_m} \le |x-y|^{1 - \beta_m} \,,
\end{aligned}
$$
for every $x,y\in A$ and $\varepsilon\le 1$. Since, by \eqref{est_nabla-u-eps}, $|\nabla u_\varepsilon|$ is also uniformly bounded in $A$, by the Arzel\`a-Ascoli theorem, there exist a subsequence, still denoted by $(u_\varepsilon)$, and a function $\bar{u} \in C^1(A)$ such that $\lim_{\varepsilon\to0}\nabla u_{\varepsilon} = \nabla \bar{u}$ in the uniform topology on $A$. By choosing $\bar{u}(0) = 0$, we obtain that $u_{\varepsilon} \to \bar{u}$ in $C^1(A)$. 
By \eqref{boundedness} and the Fatou lemma we have that $\|\nabla \bar{u}\|_{L^{2m}(B_{\bar{R}(0)})} \leq (C'')^{1/(2m)}$. Hence, for any $\psi \in [L^{2m}(B_{\bar R}(0))]^N$
\begin{multline*}
\left| \int_{B_{\bar{R}}(0)}  (|\nabla u_\varepsilon|^{2m-2}\nabla u_\varepsilon - |\nabla \bar u|^{2m-2}\nabla \bar u) \psi \, dx \right| \\
\leq 
\int_{A} \big(\big||\nabla u_\varepsilon|^{2m-2}\nabla u_\varepsilon - |\nabla \bar u|^{2m-2}\nabla \bar u\big|\big) |\psi|dx
+ 2\frac{(C')^{2m-1}}{\bar R^{N-1}}  \|\,|\psi|\,\|_{L^1(B_{\bar R}(0)\setminus A)}. 
\end{multline*}
For any $\delta > 0$ we can take $A$ such that  $\|\,|\psi|\,\|_{L^1(B_{\bar R}(0)\setminus A)} \leq \delta$ and for sufficiently small $\varepsilon > 0$
we have, from the uniform convergence of $\nabla u_\varepsilon$ on $A$, that 
\begin{equation*}
\left| \int_{ B_{\bar{R}}(0)}  (|\nabla u_\varepsilon|^{2m-2}\nabla u_\varepsilon - |\nabla \bar u|^{2m-2}\nabla \bar u) \psi \, dx \right| 
\leq C\delta \,
\end{equation*}
for some $C>0$ independent of $\varepsilon$. Since $\delta > 0$ and $\psi\in [L^{2m}(B_{\bar R}(0))]^N$ were arbitrary, we have $|\nabla u_\varepsilon|^{2m-2}\nabla u_\varepsilon\rightharpoonup |\nabla \bar{u}|^{2m-2}\nabla \bar{u}$ 
in $[L^{(2m)'}(B_{\bar R}(0))]^N$.  
Recalling that $u_\varepsilon$ solves weakly \eqref{Eeps}, we have for any $\varphi\in C^\infty_c(B_{\bar R}(0))$ 
$$
\sum_{h=1}^m \int_{B_{\bar R}(0)}\varepsilon^{N-2h+\frac{(2m-N)(2h-1)}{2m-1}}\alpha_h|\nabla u_\varepsilon|^{2h-2}\nabla u_\varepsilon\cdot\nabla \varphi dx=a_k\varphi(0)
$$
and by  passing $\varepsilon \to 0$ and using proved weak convergences,  we obtain  
$$
\int_{B_{\bar R}(0)}\alpha_m|\nabla \bar{u}|^{2m-2}\nabla \bar{u} \cdot\nabla \varphi dx=a_k\varphi(0)\,,
$$
or equivalently $\bar{u}$ is a weak solution of 
\begin{equation}\label{eq-for-bar-u}
-\alpha_m\Delta_{2m}u=a_k\delta_0\quad\mbox{in }B_{\bar R}(0). 
\end{equation}

\smallskip

{\it Step 5: Behavior of $\bar u$ and its gradient near $0$.}
By \eqref{eq-for-bar-u}, we know that $\bar u$ is $2m$-harmonic in $B_{\bar R}(0)\setminus\{0\}$ and  $\bar u(0)=0$. 
As in  Step~2,  \cite[Theorem 1.2]{Baroni} with $g(t):=\alpha_m t^{2m-1}$ yields for a.e. $x\in B_{\bar R/4}(0)$
$$
\begin{aligned}
|\nabla \bar u(x)|&\le \left\{\frac{c}{\alpha_m|x|^{N-1}}\left[\frac{|a_k|}{(N-1)}+\alpha_m\left(\frac{\bar R}{4}\right)^{N-1}\left(\frac{\|\nabla \bar u\|_{L^1(B_{\bar R}(0))}}{|B_{\bar R/4}|}\right)^{2m-1}\right]\right\}^{\frac1{2m-1}}\\
&=: C_0|x|^{\frac{1-N}{2m-1}}\\
|\bar u(x)|&\le \frac{2m-1}{2m-N} C_0|x|^{\frac{2m-N}{2m-1}} \,,
\end{aligned}
$$
where the second bound follows as in \eqref{ubd}. 
Hence, the isotropy result \cite[Remark~1.6]{Veron} (see also work by Serrin \cite{serrin1965singularities}) implies
\begin{equation}\label{baru-sol-fond}
\lim_{x\to0}\frac{\bar u(x)}{\mu(x)}=\gamma\quad\mbox{ and }\quad
\lim_{x\to0}|x|^{\frac{N-1}{2m-1}}\nabla(\bar u-\gamma\mu)=0,
\end{equation}
where $\gamma:=\mathrm{sign}(a_k)\left(\frac{|a_k|}{\alpha_m}\right)^{\frac1{2m-1}}$, and  $\mu(x):=\kappa_m(N)|x|^{\frac{2m-N}{2m-1}}$ with $\kappa_m(N):=-\frac{2m-1}{2m-N}(N|B_1|)^{-\frac1{2m-1}}$ is the fundamental solution of the $-\Delta_{2m}$.
\smallskip

{\it Step 6: Behavior of $u_m$ and its gradient near $x_k$.} Since $|x|^{\frac{N-1}{2m-1}} |\nabla \mu| =  |\kappa_m| \frac{2m-N}{2m-1}$, from \eqref{baru-sol-fond} follows
$$
\lim_{x\to0}|\nabla \bar u(x)| |x|^{\frac{N-1}{2m-1}}=\frac{2m-N}{2m-1}|\gamma\kappa_m|.
$$
Furthermore, by Step 4 we know in particular that $\nabla u_\varepsilon\to\nabla \bar u$ pointwise in $B_{\bar R}(0)\setminus\{0\}$. Hence, 
$$
\lim_{x\to0}\left(\lim_{\varepsilon\to0}|\nabla u_\varepsilon(x)||x|^{\frac{N-1}{2m-1}}\right)=\lim_{x\to0}|\nabla \bar u(x)||x|^{\frac{N-1}{2m-1}}=
\frac{2m-N}{2m-1}|\gamma\kappa_m|
$$
and by the definition of $u_\varepsilon$, 
$$
\frac{2m-N}{2m-1}|\gamma\kappa_m| = 
\lim_{x\to0}\left(\lim_{\varepsilon\to0} \varepsilon^{\frac{N-1}{2m-1}} |\nabla u_{m,k}(\varepsilon x)||x|^{\frac{N-1}{2m-1}}\right)=
\lim_{y \to 0}|\nabla  u_{m, k} (y)||y|^{\frac{N-1}{2m-1}} \,.
$$
Consequently
$$|\nabla u_{m,k}(x)|\sim|\gamma\kappa_m|\frac{2m-N}{2m-1}|x|^{\frac{1-N}{2m-1}}\quad\mbox{as }x\to0,$$
which in turn implies \eqref{grad-growth} with $K_m':= |\gamma\kappa_m| \frac{2m-N}{2m-1}$.
Analogously, by Step 4 we also know that $u_\varepsilon\to \bar u$ pointwise in $B_{\bar R}(0)\setminus\{0\}$. Therefore, by \eqref{baru-sol-fond}
$$
\lim_{x\to0}\left(\lim_{\varepsilon\to0}\frac{u_\varepsilon(x)}{\gamma \kappa_m |x|^{\frac{2m-N}{2m-1}}}\right)=1
$$
which in terms of $u_{m, k}$ gives
$$
\lim_{x\to0}\frac{u_{m,k}(x)}{|x|^{\frac{2m-N}{2m-1}}}=\gamma\kappa_m
$$
and proves \eqref{u-growth} with $K_m:=\gamma\kappa_m$. In particular, 
if $a_k>0$, then $K_m\cdot a_k<0$, and $x_k$ is a relative strict maximizer of $u_m$, while if  $a_k<0$ it is a relative strict minimizer of $u_m$.
\end{proof}

\begin{remark}\label{Kto1} Observe that, since $\alpha_m=\frac{(2m-3)!!}{(2m-2)!!}$, 
$$
\lim_{m\to\infty}|K_m|=\lim_{m\to\infty}\frac{2m-1}{2m-N}\left(\frac{|a_k|}{N|B_1|\alpha_m}\right)^{\frac1{2m-1}}=1.
$$ 
\end{remark}

\section*{Acknowledgments} The authors thank Maria Colombo for a fruitful discussion and for pointing to us the reference \cite{Baroni}.

The authors acknowledge the support of the projects MIS F.4508.14 (FNRS) \& ARC AUWB-2012-12/17-ULB1- IAPAS. 

F. Colasuonno was partially supported by the INdAM - GNAMPA Project 2017 ``Regolarit\`a delle soluzioni viscose per equazioni a derivate parziali non lineari degeneri''.

\bibliographystyle{abbrv}
\bibliography{biblio}
 
\end{document}